%% file: MirrorLangevin_ArXiv.tex
\documentclass[a4paper, 11pt]{amsart}

\title{Wasserstein Control of Mirror Langevin Monte Carlo}
	
\usepackage{mystyle_arXiv}

\author[1]{Kelvin Shuangjian Zhang$^\dagger$}\thanks{$^\dagger$ CNRS and D\'epartement de Math\'ematiques et Applications, \'Ecole Normale Sup\'erieure / Universit\'e PSL, Paris, France {\tt szhang@ens.fr}}
\author[2]{Gabriel Peyr\'e$^\ddagger$}\thanks{$^\ddagger$ CNRS and D\'epartement de Math\'ematiques et Applications, \'Ecole Normale Sup\'erieure / Universit\'e PSL, Paris, France {\tt gabriel.peyre@ens.fr}}
\author[3]{Jalal Fadili$^\mathsection$}\thanks{$^\mathsection$ Normandie Univ, ENSICAEN, UNICAEN, CNRS, GREYC, France {\tt jalal.fadili@greyc.ensicaen.fr}}
\author[4]{Marcelo Pereyra$^\mathparagraph$}\thanks{$^\mathparagraph$ School of Mathematical and Computer Sciences, Heriot-Watt University, UK  {\tt m.pereyra@hw.ac.uk}}

\allowdisplaybreaks

\begin{document}

\input{abstract}

\maketitle
\vspace{-0.5cm}
\noindent \textbf{Keywords.} {\small \textit{Riemannian Langevin Monte Carlo, Hessian manifold, sampling, contraction, Baillon-Haddad inequality.}}
\vspace{0.5cm}


\input{introduction}

\input{contributions}

\input{proof_sketch}

\input{conclusion}

\bibliographystyle{plainnat}
\bibliography{MirrorLangevinCOLT2020}

\appendix

\input{app_sufficient_condition}

\input{app_assumptions}

\input{app_Baillon_Haddad}

\input{app_proof_of_cors}

\end{document}

%% file: abstract.tex

\begin{abstract}
Discretized Langevin diffusions are efficient Monte Carlo methods for sampling from high dimensional target densities that are log-Lipschitz-smooth and (strongly) log-concave. In particular, the Euclidean Langevin Monte Carlo sampling algorithm has received much attention lately, leading to a detailed understanding of its non-asymptotic convergence properties and of the role that smoothness and log-concavity play in the convergence rate. Distributions that do not possess these regularity properties can be addressed by considering a Riemannian Langevin diffusion with a metric capturing the local geometry of the log-density. However, the Monte Carlo algorithms derived from discretizations of such Riemannian Langevin diffusions are notoriously difficult to analyze. In this paper, we consider Langevin diffusions on a Hessian-type manifold and study a discretization that is closely related to the mirror-descent scheme. We establish for the first time a non-asymptotic upper-bound on the sampling error of the resulting Hessian Riemannian Langevin Monte Carlo algorithm. This bound is measured according to a Wasserstein distance induced by a Riemannian metric ground cost capturing the Hessian structure and closely related to a self-concordance-like condition. The upper-bound implies, for instance, that the iterates contract toward a Wasserstein ball around the target density whose radius is made explicit. Our theory recovers existing Euclidean results and can cope with a wide variety of Hessian metrics related to highly non-flat geometries. 
\end{abstract}

%% file: introduction.tex

\section{Introduction}

\subsection{Problem and setting}
We consider the problem of sampling from a target probability distribution $d\pi = e^{-f(\x)} d\x$ supported on a domain $\Xx \subset \R^p$, where $f$ is differentiable on $\Xx$. We are particularly interested in sampling algorithms that scale efficiently to high dimensions. When $f$ is Lipschitz-smooth (i.e. differentiable with Lipschitz gradient) and strongly convex on $\Xx$, then the conventional Langevin Monte Carlo (LMC) algorithm derived from an Euler-Maruyama discretization of the Langevin stochastic differential equation (SDE) is one of the most computationally efficient methods to sample from $\pi$. In this paper, we endow $\Xx$ with a carefully designed Riemannian structure and study the non-asymptotic convergence properties of a Riemannian generalization of the LMC algorithm. The motivation is that by endowing $\Xx$ with an appropriate Riemannian geometry, it is possible to obtain algorithms with better convergence properties, and which can tackle distributions that are beyond the scope of the Euclidean LMC algorithm. We consider Riemannian structures of Hessian type~\citep{Shima07}; the corresponding metric is induced by the Hessian $D^2 \phi(\x)$ of some $C^2(\Xx)$ \emph{Legendre-type} convex potential/entropy $\phi$ on $\Xx$ (see \cite[Chapter~26]{Rockafellar70} for a comprehensive account on Legendre functions).

\medskip
\noindent\textbf{Discrete scheme.} In the same vein as in \cite{HsiehCevher18}, we consider a sampling analogue of mirror-descent as an extension of the classical Euler-Maruyama discretization of the Langevin SDE, which reads, starting from some random vector $\X_0$ on $\Xx$, 
\begin{flalign}\label{discrete_X}
\begin{aligned}
\X_{k+1} \eqdef \nabla\phi^*\Big(&\nabla\phi(\X_{k})-h_{k+1}\nabla f(\X_{k}) + \sqrt{2h_{k+1}[D^2\phi(\X_{k})]}{\bm \xi}_{k+1}\Big).
\end{aligned}
\end{flalign}
Here $\phi^*$ is the Legendre-Fenchel conjugate of $\phi$, i.e., $\phi^*({\bf y}) \eqdef \sup_{{\bf x} \in \Xx} \langle {\bf x}, {\bf y} \rangle - \phi({\bf x})$, $\{h_k\}_{k \in \N} \subset \RR_{++}$ is the sequence of step-sizes, and $\{{\bm \xi}_{k}\}_{k\in\N}$ is a sequence of standard normal random vectors that are mutually independent and independent of $\X_{0}$, which is either deterministic or random. Let us recall the useful fact that $\phi$ is of Legendre type if and only if its conjugate $\phi^*$ is of Legendre type. Moreover, the gradient $\nabla \phi$ of $\phi$ is a bijection from $\Int~\dom(\phi)=\Xx$ to $\Int~\dom(\phi^*)=\mathcal{Y}$ and its inverse obeys $(\nabla \phi)^{-1}=\nabla \phi^*$, see \cite[Theorem~26.5]{Rockafellar70}. Thus \eqref{discrete_X} makes perfectly sense as a single-valued mapping from $\Xx$ to $\Xx$.

In the following, we call iteration \eqref{discrete_X} {\bf Hessian Riemannian Langevin Monte Carlo (HRLMC)} algorithm. Note that~\cite{HsiehCevher18} does not study this method, and rather settles for a different discretization, which is simpler to analyze (being a change of variable applied to the Euclidean case) and enjoys theoretical guarantees that are markedly different from ours (we refer to Section~\ref{subsec:prior} for a detailed comparison). 

In the case where ${\bm \xi}_{k}=0$ (optimization framework), one recovers the mirror descent minimization algorithm~\citep{NemirovskiYudin83,BauschkeBolteTeboulle17,LuFreundNesterov17}. The classical Euclidean case is recovered when $\phi$ is the energy, i.e., $\phi(\x)=\norm{\x}^2_2/2$. Other popular options to sample in $\Xx=\RR_{++}^p$ include Shannon entropy $\phi({\x})=\sum_i x_i \log(x_i)$ and Burg's entropy $\phi({\x})=-\sum_i \log(x_i)$. 

As mentioned previously, the key motivations behind switching from Euclidean LMC methods to the HRLMC scheme are that by choosing an entropy $\phi$ adapted to $f$, one can either obtain better smoothness and strong convexity properties or even recover smoothness and strong convexity relatively to $\phi$ in cases where $f$ is neither Lipschitz-smooth nor strongly convex in the standard Euclidean geometry. The goal of this paper is to provide the first step toward a theoretical understanding of these phenomena, by establishing a non-asymptotic upper-bound on the error in a properly designed Wasserstein distance for sampling from $\pi$ using HRLMC. The terms in the bound explicitly reflect the interleaved geometries of $f$ and $\phi$.

\medskip
\noindent\textbf{Continuous flow.}
It can be shown that the HRLMC algorithm \eqref{discrete_X} can be viewed as a discretization of a Riemannian SDE. Denoting $\Y_t \eqdef \nabla\phi(\X_t)$, this SDE reads
\begin{flalign}\label{SDE_Y}
d\Y_t = -\nabla f \circ \nabla \phi^*(\Y_t) dt + \sqrt{2 [D^2 \phi^*(\Y_t)]^{-1}} d{\bf B}_t,
\end{flalign}
where $\{{\bf B}_t \}_{t\ge 0}$ is a standard $p$-dimensional Brownian motion. If moreover $\phi \in C^3(\Xx)$, then Legendreness of $\phi$ entails that the SDE on $\X_t$ reads
\begin{flalign}
\label{SDE_X_1}
	d\X_t = \big(  \theta({\X}_t) -[D^2\phi(\X_t)]^{-1} \nabla f(\X_t)\big)dt 
+ \sqrt{2[D^2\phi(\X_t)]^{-1}} d{\bf B}_t,
\end{flalign}
where the additional drift term $\theta(\X_t) \eqdef - [D^2\phi(\X_t)]^{-1}{\rm Tr}\left( D^3\phi(\X_t) [D^2\phi(\X_t)]^{-1}\right)$. Moreover, the corresponding density satisfies a Fokker-Planck equation that has $\pi$ as its stationary solution (we omit the details for the sake of brevity). When $\phi(\x)=\norm{\x}^2_2/2$, then $\X_t=\Y_t$, and \eqref{SDE_Y} and \eqref{SDE_X_1} coincide with the standard Langevin diffusion. The SDE \eqref{SDE_X_1}, viewed as Brownian motion on a Hessian manifold corrected by a Riemannian drift term $-[D^2\phi(\X_t)]^{-1} \nabla f(\X_t)dt$, is then its natural generalization to a Riemannian manifold with a Hessian structure. We will show in Appendix~\ref{app:sufficient_condition} that both \eqref{SDE_Y} and \eqref{SDE_X_1} are well-posed, under a self-concordance-like condition \eqref{assumption:A1}.

\subsection{Previous work}\label{subsec:prior}

The goal of this paper is to provide non-asymptotic upper-bounds on the Wasserstein distance, with an appropriate ground cost, between the distribution $\mu_k$ of ${\bf X}_k$ and the target distribution $\pi$.

\medskip
\noindent\textbf{Langevin Monte Carlo (LMC) under (strong) log-concavity.} 
The Euclidean LMC, corresponding to $\phi(\x)=\norm{\x}^2_2/2$, has been extensively studied in the literature, where non-asymptotic error bounds have been established under various sampling error metrics (Kullback-Leibler, Total-Variation, or Wasserstein). The case where $f$ is $m$-strongly convex with a $M$-Lipschitz gradient is the one that has been most widely studied~\citep{Dalalyan17a,Dalalyan17b,DurmusMoulines17,ChengBartlett18,DurmusMoulines19,DalalyanKaragulyan19,DurmusMajewskiMiasojedow19,DwivediEtAlYu19}. In particular,~\citep{DalalyanKaragulyan19} have shown that, when using a constant step size $h_k = h \in (0, \frac{2}{M})$, the LMC algorithm converges to the sampling distribution with a contraction factor $\rho = \max(1-mh, Mh-1)$. More precisely,
\begin{flalign}\label{eq-convergence-W2}
\begin{aligned}
W_2(\mu_k, \pi) &\le \rho^k W_2 (\mu_0, \pi) + \frac{1.65 Mh^{\frac{3}{2}}p^{\frac{1}{2}}}{1-\rho} \\
		&\le (1-mh)^k W_2 (\mu_0, \pi) + 1.65 (M/m)(ph)^{\frac{1}{2}}, \quad \text{if } h \leq 2/(m+M) ,
\end{aligned}
\end{flalign}
where  $W_2$ is the 2-Wasserstein distance between two probability measures, i.e.,
\begin{flalign*}
	W_2^2(\mu, \nu) \eqdef \inf\limits_{\X \sim \mu, \X' \sim \nu} \E\left[\norm{\X-\X'}_2^2\right].
\end{flalign*}
This is the best known result in Wasserstein distance.

\cite{DurmusMoulinesPereyra18} studied the case of non-Lipschitz-smooth (strongly) convex $f$ via Moreau-Yosida regularization, and \cite{BubeckEldanLehec18,BrosseEtAlPereyra17} the case of log-Lipschitz-smooth strongly log-concave densities supported on a convex compact set. \cite{ChengEtAlJordan18,DalalyanKineticLMC18} investigated the case of a kinetic Langevin diffusion (i.e., underdamped LMC) for the same class of densities, showing that it leads to improved dependence on the dimension and error.

Non-asymptotic sampling error bounds when $f$ is Lipschitz-smooth and merely convex (but not strongly so) have been established in the literature in KL and TV~\cite{DurmusMajewskiMiasojedow19}, and in Wasserstein distance~\cite{DalalyanConvexLMC19} for various discrete LMC schemes.

\medskip
\noindent\textbf{LMC beyond log-concavity.}
Obtaining convergence results is very difficult when $f$ is not convex. \cite{LuuFadiliChesneau17} considered densities that are neither necessarily smooth nor log-concave and provided asymptotic consistency guarantees. Assuming convexity at infinity, \cite{ChengEtAlJordan19, Majka18} obtained convergence results in the 1-Wasserstein distance by using results in \cite{Eberle16}. When replacing convexity with a dissipativity condition, a non-asymptotic bound was first provided by \cite{Raginsky17} in the 2-Wasserstein distance, and then improved by \cite{Chau19}. In \cite{Zhang19}, assumptions are further weakened by assuming only local Lipschitz continuity of $\nabla f$ and by relaxing conditions of convexity at infinity and uniform dissipativity.

\medskip
\noindent\textbf{Continuous Riemannian Langevin dynamics.}
The SDE~\eqref{SDE_X_1} is a special case of the so-called Riemannian Langevin dynamics, which appeared in \cite{RobertsStramer02, GirolamiCalderhead11, PattersonTeh13}, when considering $\Xx$ as a Riemannian manifold with Hessian metric $D^2\phi$.
For this Riemannian Langevin SDE setting, it is known since \cite{Kent78} that $\X_t$ has $\pi$ as its unique invariant measure as long as $\X_t$ is non-explosive. For the conditions on the non-explosion of diffusions, see \cite{StroockVaradhan79}. Moreover, the linear convergence theory of the corresponding Fokker-Planck equation is known since  \cite{ArnoldMarkowichToscaniUnterreiter01}, relying on the positivity of Bakry-Emery tensor; see \citep{Bakry14} for a comprehensive account.
Discretization schemes of the Riemannian Langevin SDE~\eqref{SDE_X_1} were proposed in \cite{RobertsStramer02, GirolamiCalderhead11, PattersonTeh13}. For instance, \cite{RobertsStramer02} provided a linear convergence result of the Ozaki discretization under quite stringent conditions. In particular, for the Hessian manifold, this theory requires $\phi$ to be strongly convex, which in turn restricts the sampling distribution to be strongly log-concave.

In this paper, instead, we take the Euler-Maruyama discretization of~\eqref{SDE_Y} and map the process back to $\X_k$ by the mirror map $\X_k =\nabla \phi^*(\Y_k)$. This is a key difference between our HRLMC algorithm~\eqref{discrete_X} and those proposed in \cite{RobertsStramer02,GirolamiCalderhead11,PattersonTeh13}. To the best of our knowledge, there is no proof of convergence or error bounds for such Euler-Maruyama discretization of~\eqref{SDE_Y} or~\eqref{SDE_X_1}.

\medskip
\noindent\textbf{Relation to \cite{HsiehCevher18}.} 
In 2018, \cite{HsiehCevher18} studied a mirror-type discretization of Langevin dynamics.
Though it seems that their work shares apparent similarities with ours at first glance, both their scheme and results are, however, markedly different from our HRLMC. More precisely, a key difference lies in the fact that here, we use an appropriate diffusion term entailing a Gaussian noise in the discrete scheme with iteration-dependent covariances that account for the Hessian Riemannian structure. In contrast, \cite{HsiehCevher18} adopted a standard Gaussian noise instead. Moreover, they provided the existence of good mirror maps assuming $f$ is strongly convex and gave convergence of their sampling algorithm under $1$-strongly convex mirror maps.
In this paper, we relax these requirements to relative versions and aim to generalize  results from the literature relying on strong convexity and Lipschitz-smoothness of $f$.

\subsection{Contributions}

In this paper, by relaxing strong convexity and Lipschitz-smoothness of $f$ to the relative versions with respect to a Legendre-type entropy $\phi$, we prove that, if the step-sizes $h_k$ are chosen sensibly, the law of discrete process~\eqref{discrete_X} contracts into a Wasserstein ball centered at the desired invariant distribution, whose radius is given explicitly. This Wasserstein distance relies on a ground cost, which is a Riemannian distance that captures the Hessian structure of the manifold. In fact, convergence to $\pi$ is not achieved in general unless $\phi$ is quadratic, but our bound allows us to isolate a bias term that depends on the interleaved geometries of $f$ and $\phi$. In particular, our method recovers the state-of-the-art non-asymptotic sampling error bounds in Wasserstein distance when $\phi(\x)=\norm{\x}^2_2/2$ \citep{DalalyanKaragulyan19}. 

Section~\ref{section:contraction} states the main contribution of this paper, Proposition~\ref{cor:constant_step_size}, whose proof relies on a more general result (Theorem \ref{thm:contractibility}) detailed in Section~\ref{section:sketch_proof}. In the appendices, we collect all details of the discussions and proofs. This includes discussions of our assumptions (e.g., intuition behind condition \eqref{assumption:A1}, relation of~\eqref{assumption:A3} and \eqref{assumption:A4} to relative strong convexity and relative smoothness). We also present a generalized Baillon-Haddad inequality~\eqref{eqn:Baillon-Haddad_extend} that is of independent interest, and give the detailed proofs of Proposition~\ref{cor:constant_step_size}, Corollary~\ref{cor:contraction_ball}, and Proposition~\ref{prop:main_proof_phi}.

\medskip
\noindent\textbf{Notations.}
Thought out the paper, $\mathcal{M}_{k \times l}$ is the ring of $k \times l$ matrices on $\RR$. $\norm{{\bf v}}_2$ is the Euclidean norm of a vector ${\bf v}$; for a matrix ${\bf M} \in \mathcal{M}_{k \times l}$, $\norm{{\bf M}}_2$ stands for its spectral norm. That is, $\norm{{\bf M}}_2 = \sqrt{\lambda_{\max}({\bf M}^T{\bf M})}$, where $\lambda_{\max}$ represents the largest value of eigenvalues. 
By definition, $\norm{{\bf M}}_2\le \delta$ is equivalent to ${\bf M}^T{\bf M}\preceq \delta^2 {\bf I}_{p}$, i.e., ${\bf M}^T{\bf M} - \delta^2 {\bf I}_{p}$ is negative semi-definite. Another matrix norm we use here is the Frobenius norm $\norm{{\bf M}}_F = \sqrt{\sum_{i,j = 1} {\bf M}_{ij}^2} = \sqrt{{\rm Tr}({\bf M}^T{\bf M})}$, where ${\rm Tr}$ is the trace operator.
The commutator of two square matrices ${\bf M}_1, {\bf M}_2 \in \mathcal{M}_{p\times p}$ is denoted as $[{\bf M}_1, {\bf M}_2] \eqdef {\bf M}_1{\bf M}_2 - {\bf M}_2 {\bf M}_1$.

%% file: contributions.tex
\bigskip
\section{Main contributions}\label{section:contraction}

In this section, we state our main contributions, namely that the HRLMC algorithm~\eqref{discrete_X} contracts into a Wasserstein ball centered at the invariant measure.

\subsection{Assumptions on $\phi$ and $f$} 

In the following, we assume that the domain $\Xx\subset \R^p$ is open, contractible and $\nabla \left(\frac{d\pi}{d \x}\right) = 0$ on its boundary $\partial \Xx$. 
To avoid technical issues, we assume that both $f$ and $\phi$ are in $C^{3}(\Xx)$ and $\phi$ is of Legendre type.

\medskip\noindent\textbf{Self-concordance-like condition on $\phi$.}
Our first condition imposes the existence of $\kappa\ge 0$ such that 
\begin{equation}\tag{\text{\bf A1}} \label{assumption:A1}
\forall (\x, \x') \in \Xx^2, \quad
\sqrt{2}\norm{D^2 \phi(\x)^{\frac{1}{2}} - D^2 \phi(\x')^{\frac{1}{2}}}_F  \le 
\kappa \norm{\nabla \phi(\x) - \nabla \phi(\x')}_2.
\end{equation}

In 1D, it is easy to check that this condition is equivalent to self-concordance. The general case is more intricate. \eqref{assumption:A1} is important to guarantee the existence and uniqueness of the strong solution of continuous dynamics~\eqref{SDE_Y} (see \citep[Theorem~5.2.1]{Oksendal03}). In fact, if it is violated, the Lipschitz condition of the SDE also fails, which removes the general theoretical guarantee for~\eqref{SDE_Y} to have an unique solution. See Appendix~\ref{app:sufficient_condition} for further details.

\medskip\noindent\textbf{Moment condition on the Hessian of $\phi$.}
The second constant involved in our analysis is
\begin{equation}\tag{\text{\bf A2}}\label{assumption:A2}
	R \eqdef \E_{\X \sim \pi} \left[\norm{D^2\phi(\X)}_2\right] = \int_\Xx \norm{D^2\phi(\x)}_2 e^{-f(\x)} \text{d} \x < + \infty.
\end{equation}

\medskip\noindent\textbf{Relative strong convexity and Lipschitz-smoothness.}
In this paper, we relax the usual Euclidean condition of strong convexity and Lipschitz-smoothness as follows: there exists $m \ge 0$, $M > 0$ such that $\forall (\x,\x') \in \Xx^2$, 
\begin{align}\tag{\text{\bf A3}}\label{assumption:A3}
	m\norm{\nabla\phi(\x)-\nabla\phi(\x')}_2^2 &\le 
		\langle \nabla f(\x) - \nabla f(\x'), \nabla\phi(\x) -\nabla\phi(\x') \rangle; \\
	\tag{\text{\bf A4}}
\label{assumption:A4}	\norm{\nabla f(\x) -\nabla f (\x')}_2 &\le M \norm{\nabla\phi(\x) - \nabla\phi(\x')}_2.
\end{align}
In the Euclidean case when $\phi(\x)=\norm{\x}^2/2$, one recovers the usual notion of strong convexity of $f$ and Lipschitz continuity of its gradient. 
The condition~\eqref{assumption:A3} and~\eqref{assumption:A4} imply, respectively, the relative strong convexity and relative Lipschitz-smoothness defined in \cite{LuFreundNesterov17,BauschkeBolteTeboulle17}. More precisely, they imply that
$		m D^2 \phi(\x) \preceq D^2 f(\x) \preceq M D^2 \phi(\x), 
\text{ for all } \x\in \Xx.$
The converse is not true in general. See details in Appendix~\ref{app:assumptions}.

\medskip\noindent\textbf{Bound on the commutator of $D^2\phi$ and $D^2 f$.} Whenever the Hessians $D^2f$ and $D^2\phi$ do not commute, we require the following assumption to quantify the commutator:
\begin{equation}\tag{\text{\bf A5}}\label{assumption:A5}
	\exists \delta \ge 0, \;
	\forall \x \in \Xx, \quad  
	\norm{ \left[(D^2\phi(\x))^{-1}, D^2 f(\x)\right]}_2 \le \delta.
\end{equation}
This control is crucial to prove the generalized Baillon-Haddad inequality (Proposition \ref{prop:Baillon-Haddad_extend_weak}).

\subsection{Wasserstein Distance}

While the de-facto geodesic distance on $\Xx$ endowed with the Hessian structure is the Riemannian distance associated with $D^2 \phi(\x)$, this distance cannot be computed in closed form. We thus settle for a simpler one, which is the Riemannian distance $d$ associated with the squared Hessian $[D^2 \phi(\x)]^2$. One can check that the diffeomorphism $\nabla \phi: (\Xx, [D^2 \phi(\x)]^2) \rightarrow (\mathcal{Y}, \mathbf{I}_p)$ is an isometry (see \citep[Chapter~1]{doCarmo92} for a detailed account on the isometry of Riemannian manifolds). Therefore,  $d(\x, \x') = \norm{\nabla\phi(\x)-\nabla\phi(\x')}_2$  for any $\x, \x'\in \Xx$.

With this ground distance, the natural associated geometric distance on the space of probability distributions on $\Xx$ is the Wasserstein distance
\begin{flalign}\label{def:Wasserstein_distance}
W_{2,\phi}^2(\mu, \nu) \eqdef \inf_{\x\sim \mu, \x'\sim \nu} \E\left[d^2(\x, \x')\right] 
= \inf_{\x\sim \mu, \x'\sim \nu} \E\left[\norm{\nabla\phi(\x)-\nabla\phi(\x')}_2^2\right].
\end{flalign}
When $\phi(\x)=\norm{\x}^2/2$, one recovers the usual $W_2$ distance used in~\eqref{eq-convergence-W2}.

\subsection{Statement of the main result}

From now on, we assume that conditions~\eqref{assumption:A1}--\eqref{assumption:A5} are satisfied. Denote by $\mu_k$ the law of the random vector $\X_{k}$ defined in~\eqref{discrete_X} and define
\begin{equation*}
\tilde{ \kappa} \eqdef \sqrt{\kappa^2 + \frac{\delta(4M+\delta)}{2(m+M)}}.
\end{equation*}
Our main contribution is Theorem~\ref{thm:contractibility}, whose statement and proof will be given shortly in a forthcoming section.
For the sake of clarity, we first apply it below to the case of constant step sizes, which makes it easier to get the gist of our main result and compare it with existing works.

\begin{proposition}[Constant step size]\label{cor:constant_step_size}
	Assume conditions~\eqref{assumption:A1}--\eqref{assumption:A5} are satisfied with $\tilde{ \kappa}	< \sqrt{2m}$ and $h_k = h  < \min\left(\frac{2m-\tilde{ \kappa}^2}{m^2},\frac{2M-\tilde{ \kappa}^2}{M^2}\right)$.
	Then 
	\begin{flalign}
	\begin{aligned}
	W_{2,\phi}(\mu_{k}, \pi) \le  \rho^k W_{2,\phi}(\mu_{0}, \pi) + h^{\frac{3}{2}}p^{\frac{1}{2}}(1-\rho)^{-1}\beta_2(R, M, \kappa) + hp^{\frac{1}{2}}(1-\rho)^{-1}\beta_1(R, \kappa),
	\end{aligned}
	\end{flalign}	
	where $\rho \eqdef \max\left(\sqrt{(1-mh)^2 + h \tilde{ \kappa}^2}, \sqrt{(1-Mh)^2 + h\tilde{ \kappa}^2}\right)<1$,
	$\beta_1(R, \kappa) \eqdef \kappa R^{\frac{1}{2}}$, and $\beta_2(R, M, \kappa) \eqdef M^{\frac{1}{2}} R^{\frac{1}{2}}\left(\frac{7\sqrt{2M}}{6} +  \frac{\kappa}{\sqrt{3}} \right)$ are dimension-free constants.
\end{proposition}

The error upper-bound is composed of three terms. The first one comes from the time finiteness that decreases exponentially, while the second corresponds to the discretization error. These two terms are standard in LMC. The last term is new and reveals the price to be paid if one trades the standard strong convexity and Lipschitz-smoothness for their relative versions in the Riemannian geometry induced by $\phi$. If $h$ is sufficiently small, one can see that $(1-\rho)^{-1} = \mathcal{O}(h^{-1})$, where the constant in the order depends on $(m,M,\kappa,\delta)$. In turn, the discretization error term will scale as $\mathcal{O}(\beta_2(R, M, \kappa)p^{1/2}h^{1/2})$, which vanishes as $h \to 0$, while the last term is $\mathcal{O}(\beta_1(R, \kappa)p^{1/2})$. The latter is a bias term. 
Though we have no proof so far, we conjecture that this bias is unavoidable in general. Our analysis recovers exactly the particular case when $f$ is $m$-strongly convex and has an $M$-Lipschitz continuous gradient, where it satisfies conditions  ~\eqref{assumption:A1}--\eqref{assumption:A5} with $\phi(\x)=\norm{\x}^2/2$, $\kappa=0$, $R =1$, $\delta=0$, $\beta_1 = 0$, $\beta_2 = \frac{7\sqrt{2}M}{6}$, $\tilde{ \kappa} = 0$, $\rho= \max\{1-mh, Mh -1\}$, and $W_{2,\phi} = W_{2}$. 
Thus the bias term vanishes and Proposition~\ref{cor:constant_step_size} recovers the sampling error bound of LMC from \citep[Theorem~1]{DalalyanKaragulyan19}, recalled in \eqref{eq-convergence-W2} .
Besides, our proposition covers new cases not known in the literature as we now show.

\subsection{Examples}

In this section, we provide two tables to include some examples that satisfy the assumptions \eqref{assumption:A1}--\eqref{assumption:A5} with explicit parameters.
As $\kappa$ is the only constant that depends merely on $\phi$, Table \ref{table:kappa} presents a list of entropy functions that satisfy \eqref{assumption:A1} or not, while Table \ref{table:B} gives the constants involving interplay between $\phi$ and $f$. For instance, in the example of Gamma distribution (Table \ref{table:B}, middle column), one can see clearly how dimension enters the game via $m$ and $M$.

\begin{enumerate}

\item[1.] More generally, $\phi(\x) = \sum_{i=1}^{p} \phi_i(x_i)$ satisfies \eqref{assumption:A1} with $\kappa = \sqrt{2}M'$ provided that $[(\phi_i^*)'']^{-\frac{1}{2}}$ has an $M'$-Lipschitz continuous gradient for each $i$. If $f(\x) = \sum_{i=1}^p f_i(x_i)$, then it satisfies \eqref{assumption:A2} and \eqref{assumption:A5} with $R \le \sum_{i} \E_{\x\sim \pi} [\phi_i''(x_i)] $ and $\delta = 0$. Besides, \eqref{assumption:A3} and \eqref{assumption:A4} are satisfied if, for each $i$, $f_i$ is $m$-strongly convex and has an $M$-Lipschitz continuous gradient relatively to $\phi_i$, in the sense of \cite{LuFreundNesterov17}.

\item[2.] Boltzmann-Shannon entropy: When $\phi(\x) = \sum_{i=1}^{p} x_i\ln(x_i)$, however, condition \eqref{assumption:A1} is violated on $\R_{++}^p$.
\end{enumerate}

\begin{table}
	\centering
	\caption{Common entropy functions and the corresponding $\kappa$ in \eqref{assumption:A1}}
	\begin{tabular}{ |c| c| c|}
		\hline
		$\phi$ & $\kappa$  &Domain\\ 
		\hline
		$\norm{\x}^2/2$ & 0 & $\R^p$ \\  
		\hline
		$- \sum_{i} \ln(x_i)$ & $\sqrt{2}$  & $\R_{++}^p$\\ 
		\hline
		$\sum_{i} x_i \ln(x_i)$ & $\infty$  & $\R_{++}^p$\\ 
		\hline
		$ -\ln(x)-\ln(1-x)$  & $\sqrt{2}$  & $(0,1)$\\ 
		\hline 
		 $\sum_{i}a_i x_i \ln(x_i)-\sum_{i} (1-a_i)\ln(x_i)$     & $\sqrt{\frac{2}{1-\max_i a_i}}$ & {$\R_{++}^p$; $a_i\in[0,1]$ }\\ 
		\hline
		$ (1-x^2)^{-1}$     & $1.43$ & $(-1,1)$\\ 
		\hline
		$  -\ln(x_2^2-x_1^2)$				&	$\sqrt{2}$	& $\{(x_1, x_2): \vert x_1\vert< x_2\}$ \\
		\hline
		$ -\ln(1-x^2)$     &	$\sqrt{2}$	&  $(-1,1)$ \\
		\hline
	\end{tabular}
\label{table:kappa}
\end{table}

\begin{table}
	\centering
	\caption{Other parameters in the assumptions~\eqref{assumption:A2}--\eqref{assumption:A5}}
	\begin{tabular}{ |c |c| c| c|}
		\hline
		&$\phi = \norm{\x}^2/2$\hspace{0.86cm}\ ~& $\phi = -\sum_{i=1}^{p} \ln(x_i)$\hspace{0.44cm}\ ~& $\phi = -\ln(x)-\ln(1-x)$\hspace{0.56cm}\ ~\\ 
		& $f = \x^T {\bf A}\x/2 + C$  & $f = \sum_{i}(1-a_i)\ln(x_i)$& $f =(1-a_1)\ln(x)$\hspace{1.4cm}\ ~ \\
		& (${\bf A}^T = {\bf A}$) & ~\hspace{0.58cm} $ + b_i x_i+C$&~\hspace{0.72cm}$ + (1-a_2)\ln(1-x)+C$\\
		\hline
		$R$ & $1$& $\sum_i (a_i-3)!/b_i^{a_i-2}$ & $\frac{(a_1-3)!(a_2-1)! + (a_1-1)!(a_2-3)!}{(a_1 + a_2-3)!}$ \\  
		\hline
		$m$ & $\lambda_{\min}({\bf A})$ & $\min_i\{a_i-1\} $  & $\min\{a_1-1, a_2 -1\}$\\ 
		\hline
		$M$ & $\lambda_{\max}({\bf A})$  & $\max_i\{a_i-1\}$  & $\max\{a_1-1, a_2 -1\}$\\ 
		\hline
		$\delta$& 0     & 0 & 0 \\ 
		\hline
	\end{tabular}
\label{table:B}
\end{table}

%% file: proof_sketch.tex

\bigskip
\section{Proof of the Main Result}\label{section:sketch_proof}
	
	\medskip	
	
\subsection{A general non-asymptotic error bound}

We are now in position to state our main theorem.

\begin{theorem}[Contractibility]\label{thm:contractibility}
Assume that~\eqref{assumption:A1}--\eqref{assumption:A5} hold such that $\tilde{ \kappa} < \sqrt{2m}$. Suppose $h_{k+1} < \min\left(\frac{2m-\tilde{ \kappa}^2}{m^2},\frac{2M-\tilde{ \kappa}^2}{M^2}\right) $. Then 
	\begin{flalign}\label{eqn:main_contraction}
	\begin{aligned}
		W_{2,\phi}(\mu_{k+1}, \pi) \le \rho_{k+1} W_{2,\phi}(\mu_{k}, \pi) + h_{k+1} p^{\frac{1}{2}}\beta_1(R, \kappa)
		+ h_{k+1}^{\frac{3}{2}}p^{\frac{1}{2}} \beta_2(R, M, \kappa).
	\end{aligned}
	\end{flalign}
	Here $\rho_{k+1} \eqdef \max\left(\sqrt{(1-mh_{k+1})^2 + h_{k+1} \tilde{ \kappa}^2}, \sqrt{(1-Mh_{k+1})^2 + h_{k+1}\tilde{ \kappa}^2}\right) <1$, $\beta_1(R, \kappa) = \kappa R^{\frac{1}{2}}$, and $\beta_2(R, M, \kappa) = M^{\frac{1}{2}} R^{\frac{1}{2}}\left(\frac{7\sqrt{2M}}{6} +  \frac{\kappa}{\sqrt{3}} \right)$ are dimension-free constants.
\end{theorem}

The main arguments to prove Theorem~\ref{thm:contractibility} will be given in Section~\ref{subsection:Baillon-Haddad} and \ref{subsec:proofthm}. This result implies in particular Proposition~\ref{cor:constant_step_size} when the step-sizes are constant. Besides, the result in \eqref{eqn:main_contraction} is invariant in scalings like $\tilde{\phi} = \alpha \phi$ for any $\alpha>0$.

Theorem~\ref{thm:contractibility} has the next corollary. In a nutshell, this corollary states that with vanishing step-sizes, the HRLMC algorithm contracts toward
a Wasserstein ball centered at the target distribution $\pi$ with radius $r_0$. The explicit formula of this radius is $r_0\eqdef \frac{2\kappa p^{\frac{1}{2}}R^{\frac{1}{2}}}{2m-\tilde{ \kappa}^2}$, which scales as $\mathcal{O}(p^{\frac{1}{2}})$ in the dimension. Moreover, once entering the ball, the distribution $\mu_k$ never leaves it. When $\phi=\norm\x^2/2$, it is clear that $r_0 = 0$ and therefore the algorithm converges to the stationary distribution. 

In the following, we use the notation $\mathcal{B}_{r}(\pi) \eqdef \{\mu \in \mathcal{P}(\Xx)| W_{2,\phi}(\mu, \pi) < r \}$ and $\overline{\mathcal{B}}_{r}(\pi) \eqdef \{\mu \in \mathcal{P}(\Xx)| W_{2,\phi}(\mu, \pi) \le r \}$, where $\mathcal{P}(\Xx)$ is the space of probability distributions on $\Xx$.

\begin{corollary}[Contracting to a Wasserstein ball]\label{cor:contraction_ball}
	Assume~\eqref{assumption:A1}--\eqref{assumption:A5} hold with $\tilde{ \kappa} 	< \sqrt{2m}$. 
	Then the following statements hold:
	\begin{enumerate}
		\item[(i)] For any $\mu_0 \in \mathcal{P}(\mathcal{X})$, there exist step-sizes 
		$\{h_{k}\}_{k \in \N}$ such that 		
		$\limsup\limits_{k \rightarrow \infty} W_{2,\phi}(\mu_{k}, \pi) 
		\le  r_0$.
		\item[(ii)] If $\mu_k \notin \overline{\mathcal{B}}_{r_0}(\pi)$, then there exists a step-size $h_{k+1}$ such that $W_{2,\phi}(\mu_{k+1}, \pi) <  W_{2,\phi}(\mu_{k}, \pi)$. 
		\item[(iii)] If $\mu_k \in \mathcal{B}_{r_0}(\pi)$,  then there exists  $h_{k+1}>0$ such that $\mu_{k+1} \in \mathcal{B}_{r_0}(\pi)$.
		\item[(iv)] If $\mu_k \in \overline{\mathcal{B}}_{r_0}(\pi) \setminus\mathcal{B}_{r_0}(\pi)$, then for any $r > r_0$, there exists $h_{k+1}>0$, such that $\mu_{k+1} \in \mathcal{B}_{r}(\pi)$.
	\end{enumerate}	
\end{corollary}

The proof can be found in Appendix~\ref{app:proof_of_corollaries} where we also construct an example of appropriate vanishing step-sizes $\{h_k\}_{k\in\N}$ that are in the order of $\frac{1}{k}$, and which guarantees that the claims of Corollary~\ref{cor:contraction_ball} hold.

\medskip\noindent\textbf{Iteration complexity bounds.}
From these guarantees, for any $\varepsilon > 0$ small enough, we can now derive the smallest number of iterations $K_\varepsilon$ (i.e., iteration complexity bound), such that the corresponding upper-bound of HRLMC with constant step-size is smaller than $r_0+\varepsilon$ after $K_\varepsilon$ steps. 
	More precisely, for any $\varepsilon$ such that 
	$0<\varepsilon<\min\left(\frac{4\sqrt{2}p^{\frac{1}{2}}\beta_2}{m\sqrt{2m-\tilde{ \kappa}^2}}, \frac{2\tilde{ \kappa}^2 p^{\frac{1}{2}}\beta_1}{(2m-\tilde{ \kappa}^2)^2}, \frac{32p^{\frac{1}{2}} \beta_2^2}{\tilde{ \kappa}^2(4m-\tilde{ \kappa}^2)^2\beta_1}\right)$, the number of iterations needed to get $W_{2,\phi}(\mu_{k}, \pi) < r_0 + \varepsilon$ with constant step-size is 
	\[
	K_\varepsilon \gtrsim \frac{pM R \left(\sqrt{M} +  \kappa \right)^2}{(2m-\tilde{ \kappa}^2)^3}\frac{1}{\varepsilon^2} \ln\left(\frac{1}{\varepsilon}\right) .
	\] 
	In the case when $\kappa=0$, it gives 
	\[
	K_\varepsilon \gtrsim \frac{p(m+M)^3 M^2 R}{(4m^2 + 4M(m-\delta) - \delta^2)^3}\frac{1}{\varepsilon^2} \ln\left(\frac{1}{\varepsilon}\right) . 
	\]
	In the classical case when $f$ is $m$-strongly convex and has an $M$-Lipschitz continuous gradient, 
	the bound becomes 
	\[
	K_\varepsilon \gtrsim \frac{pM^2}{m^3\varepsilon^2}\ln\left(\frac{1}{\varepsilon}\right),
	\] 
	which coincides with the best result in the literature of Euler-Maryuama LMC 
	(See \cite[Table~1]{DurmusMajewskiMiasojedow19} for an overview).

\medskip
\subsection{Baillon-Haddad type inequality}\label{subsection:Baillon-Haddad} 

Baillon and Haddad showed that if the gradient of a convex and continuously differentiable function is nonexpansive, then it is firmly nonexpansive (\cite{BaillonHaddad77}). This is one of the critical  steps in the proof of convergence when $\phi = \norm{\x}^2/2$. We extend the Baillon-Haddad theorem to the case of relative Lipschitz-smoothness~\eqref{assumption:A4}. We state a weaker version here, which is sufficient for the proof of the main theorem, and defer a stronger version with complete proof to the Appendix~\ref{app:Baillon-Haddad_extend}, which is of independent interest.

\medskip	

\begin{proposition}[Baillon-Haddad extension]\label{prop:Baillon-Haddad_extend_weak}
	Assume $f$ satisfies assumptions \eqref{assumption:A3}-\eqref{assumption:A5}, then for any $\x_1, \x_2 \in \mathcal{X}$,
	\begin{flalign}\label{eqn:Baillon-Haddad_extend}
	\begin{aligned}
	&\langle\nabla f(\x_1) - \nabla f(\x_2), \nabla\phi(\x_1) - \nabla\phi(\x_2)\rangle \\
	\ge & A\norm{\nabla f(\x_1) - \nabla f(\x_2)}_2^2 
		 + B\norm{\nabla\phi(\x_1) - \nabla\phi(\x_2)}_2^2,
	\end{aligned}
	\end{flalign}
	where the constants are $A \eqdef \frac{1}{m+M}$ and $B \eqdef \frac{4mM -4M\delta - \delta^2}{4(m+M)}$.
\end{proposition}
\begin{proof}[Proof sketch of Proposition~\ref{prop:Baillon-Haddad_extend_weak}]
	Denote  ${\bf A}({\bf y})\eqdef D^2 f(\nabla\phi^*({\bf y}))$ and ${\bf B}({\bf y})\eqdef D^2 \phi^*({\bf y})$.
	By the Poincar\'{e} lemma, since $d\left(\frac{1}{2}\sum_{i,j} \left[ ({\bf A}{\bf B})_{ji} - ({\bf A}{\bf B})_{ij} \right] dy_i \wedge dy_j\right)	= 0$, there exists a vector field ${\bf g}:\mathcal{Y}\rightarrow \R^p$ such that $\nabla {\bf g} = \frac{1}{2} [{\bf B}, {\bf A}].$ Define $ {\bf \tilde{g}}\eqdef \nabla f\circ\nabla\phi^* -{\bf g}$. By Stokes-Cartan theorem, ${\bf \tilde{g}}$ is path-independent. Thus, there exists a function $\tilde{f}$ on $\mathcal{Y}$ such that $\nabla \tilde{f} = {\bf \tilde{g}} = \nabla f\circ \nabla \phi^* - {\bf g}$.

	So, $D^2 \tilde{f}	= \frac{1}{2}[{\bf A}{\bf B}+ {\bf B}{\bf A}].$
	Thus, 
	by assumption \eqref{assumption:A3}-\eqref{assumption:A4}, there exist $0 \le m \le M$ such that 
	$m {\bf I}_p \preceq D^2 \tilde{f} ({\bf y}) \preceq M {\bf I}_p$ for all ${\bf y}\in \mathcal{Y}$. By the classical Baillon-Haddad theorem, we know 
	\begin{flalign*}
			&\langle\nabla\tilde{f}({\bf y}_1) - \nabla\tilde{f}({\bf y}_2), {\bf y}_1 - {\bf y}_2\rangle \\
		\ge & \frac{1}{m+M}\norm{\nabla \tilde{f}({\bf y}_1) - \nabla \tilde{f}({\bf y}_2)}_2^2 + \frac{mM}{m+M}\norm{{\bf y}_1 - {\bf y}_2}_2^2 \\
		\ge & \frac{1}{m+M}\norm{\nabla f\circ \nabla\phi^*({\bf y}_1) - \nabla f\circ \nabla\phi^*({\bf y}_2)}_2^2 - \frac{1}{m+M}\norm{ {\bf g}({\bf y}_1)- {\bf g}({\bf y}_2)}_2^2\\
		&    - \frac{2}{m+M}\langle\nabla \tilde{ f}({\bf y}_1) - \nabla \tilde{ f}({\bf y}_2), {\bf g}({\bf y}_1)- {\bf g}({\bf y}_2)\rangle + \frac{mM}{m+M}\norm{{\bf y}_1 - {\bf y}_2}_2^2.
	\end{flalign*}

	Since $\nabla {\bf g}$ is anti-symmetric, $\langle {\bf g}({\bf y}_1)  - {\bf g}({\bf y}_2), {\bf y}_1 - {\bf y}_2 \rangle = 0$ for any ${\bf y}_1, {\bf y}_2 \in \mathcal{Y}$.
	By assumption \eqref{assumption:A4} and \eqref{assumption:A5}, 
	$\langle\nabla \tilde{ f}({\bf y}_1) - \nabla \tilde{ f}({\bf y}_2), {\bf g}({\bf y}_1)- {\bf g}({\bf y}_2)\rangle 
	\le \frac{1}{2}M\delta \norm{{\bf y}_1 - {\bf y}_2}_2^2.$
	Similarly, $\norm{{\bf g}({\bf y}_1)-{\bf g}({\bf y}_2)}_2^2  \le 
	\frac{\delta^2}{4} \norm{{\bf y}_1 - {\bf y}_2}_2^2$.
	Combining all the equality and equalities above, \begin{flalign*}
		&\langle\nabla f\circ \nabla\phi^*({\bf y}_1) - \nabla f\circ \nabla\phi^*({\bf y}_2), {\bf y}_1 - {\bf y}_2\rangle \\
		= &\langle\nabla\tilde{f}({\bf y}_1) - \nabla\tilde{f}({\bf y}_2), {\bf y}_1 - {\bf y}_2\rangle \\
		\ge &
		\frac{1}{m+M}\norm{\nabla f\circ \nabla\phi^*({\bf y}_1) - \nabla f\circ \nabla\phi^*({\bf y}_2)}_2^2 + \frac{4mM -4M\delta - \delta^2}{4(m+M)}\norm{ {\bf y}_1- {\bf y}_2}_2^2.
	\end{flalign*}
	By change of variable, this implies~\eqref{eqn:Baillon-Haddad_extend}.
\end{proof}

\medskip
\subsection{Proof of Theorem \ref{thm:contractibility}}\label{subsec:proofthm}
We first recall two lemmas and state a proposition that is useful in this section. The proof of Proposition \ref{prop:main_proof_phi} is postponed to Appendix \ref{app:proof_of_corollaries}. The It\^o's isometry theorem can be found, for instance, in {\citep[Corollary~3.1.7]{Oksendal03}} for the one-dimensional case. Here we state its apparent consequence in the multidimensional case.
	
	\medskip	
	
	\begin{lemma}[It\^o's isometry] \label{lemma:Ito's_isometry} 
		Let ${\bf B}: [0, T]\times \Omega \rightarrow \R^p$ be the standard $p$-dimensional Brownian motion and ${\bf M}:  [0, T]\times \Omega \rightarrow \R^{p\times p}$ be a matrix-valued stochastic process adapted to the natural filtration of the Brownian motion. Then
		\begin{flalign}
		\E \left[ \norm{\int_0^T {\bf M}_t d{\bf B}_t}_2^2\right] = \E \left[ \int_0^T \norm{{\bf M}_t}_F^2 dt\right], 
		\end{flalign}
		whenever the integrals make sense.
	\end{lemma}

\medskip	

\begin{lemma}[Minkowski's integral inequality, {\cite[Appendix~A]{Stein70}}] \label{lemma:Minkowski_inequality}
	Suppose that $(\mathcal{S}_1, \pi_1)$ and $(\mathcal{S}_2, \pi_2)$ are two $\sigma$-finite measure spaces, $l\ge 1$ and $f: \mathcal{S}_1 \times \mathcal{S}_2 \rightarrow \R_{+}$ is measurable, then \begin{flalign}
	\left\{	\int_{\mathcal{S}_1} \left( \int_{\mathcal{S}_2} f({\bf x},{\bf y}) d\pi_2({\bf y})\right)^{l} d\pi({\bf x})\right\}^{\frac{1}{l}} \le \int_{\mathcal{S}_2}	\left( \int_{\mathcal{S}_1} f^l({\bf x},{\bf y}) d\pi({\bf x})	\right)^{\frac{1}{l}}	d\pi_2({\bf y}).
	\end{flalign} 
\end{lemma}

\medskip	

\begin{remark}
	Assume the same conditions as above, and $f_i: \mathcal{S}_1 \times \mathcal{S}_2 \rightarrow \R_{+}$ are measurable for $i = 1, ... p$, then \begin{flalign}
	\left\{	\int_{\mathcal{S}_1}  \sum_{i=1}^p\left( \int_{\mathcal{S}_2} f_i({\bf x},{\bf y}) d\pi_2({\bf y})\right)^{l} d\pi({\bf x})\right\}^{\frac{1}{l}} \le \int_{\mathcal{S}_2}	\left(  \int_{\mathcal{S}_1}\sum_{i = 1}^p f_i^l({\bf x},{\bf y}) d\pi({\bf x})	\right)^{\frac{1}{l}}	d\pi_2({\bf y}).
	\end{flalign} It can be viewed as Minkowski's inequality applying on $(\mathcal{S}_1\times\{1, ..., p\}, \pi_1 \times \pi_3)$ and $(\mathcal{S}_2, \pi_2)$, where $\pi_3$ is uniform measure up to a constant multiplication.
\end{remark}
	
\medskip	
	
\begin{proposition}\label{prop:main_proof_phi}
	Let ${\bf L}_0$ be any random vector drown from $\pi$ and ${\bf L}_t$ be a continuous dynamics satisfying \eqref{eqn:continuous_flow_L}. Then for any $s>0$, one has
	\begin{flalign}\label{eqn:main_proof_phi'}
	\sqrt{\E \left[\norm{\nabla \phi({\bf L}_0) - \nabla \phi ({\bf L}_s)}_2^2\right]}
	\le  s\sqrt{MpR} + \sqrt{2spR}.
	\end{flalign}
\end{proposition}

\medskip

{\bf\noindent Proof of Theorem \ref{thm:contractibility}.}

For notation simplicity, we use  $h$, and $\rho$ to represent $h_{k+1}$, and $\rho_{k+1}$, respectively.
Let ${\bf L}_0$ be a random vector drown from $\pi$ such that $W_{2,\phi}^2(\mu_k, \pi) = \E \left[\norm{\nabla\phi({\bf L}_0) -\nabla\phi(\X_k)}_2^2\right]$. Let ${\bf B}_t = \sqrt{t}{\bm \xi}_{k+1}$, independent of $(\X_k, {\bf L}_0)$. Define a stochastic process ${\bf L}$ such that 
\begin{flalign}\label{eqn:continuous_flow_L}
	\nabla\phi({\bf L}_t) = \nabla\phi({\bf L}_0)-\int_0^t\nabla f({\bf L}_s) ds + \sqrt{2} \int_0^t[D^2\phi({\bf L}_s)]^{\frac{1}{2}} d{\bf B}_s.
\end{flalign}
Then, by \eqref{assumption:A1}, $\{{\bf L}_t: t\ge 0\}$ has $\pi$ as its stationary distribution and 
${\bf L}_t \sim \pi$ for all $t>0$. On the other hand, our HRLMC algorithm reads
\begin{flalign}
	\nabla\phi(\X_{k+1}) = \nabla\phi(\X_{k})-h\nabla f(\X_{k}) + \sqrt{2h[D^2\phi(\X_{k})]}{\bm \xi}_{k+1}.
\end{flalign}

Let\begin{flalign*}
	{\bf A}\eqdef& \nabla\phi({\bf L}_0) -\nabla\phi(\X_k) - h (\nabla f({\bf L}_0) - \nabla f(\X_k)),\\
	{\bf C} \eqdef& \int_0^h \left(\nabla f ({\bf L}_0) - \nabla f({\bf L}_s)\right) ds,\\
	{\bf G} \eqdef& \sqrt{2h}\left([D^2\phi({\bf L}_0)]^{\frac{1}{2}} - [D^2\phi(\X_k)]^{\frac{1}{2}}	\right){\bm \xi}_{k+1},\\
	{\bf H} \eqdef& \sqrt{2}\int_0^h \left([D^2\phi({\bf L}_s)]^{\frac{1}{2}} - [D^2\phi({\bf L}_0)]^{\frac{1}{2}}\right) d{\bf B}_s.
\end{flalign*}

By definition of $W_{2,\phi}^2$ and triangular inequality, one has
\begin{flalign}\label{eqn:W2_total}
\begin{aligned}
W_{2,\phi}(\mu_{k+1}, \pi) 
\le& \sqrt{\E\left[\norm{\nabla\phi({\bf L}_h)- \nabla\phi(\X_{k+1})}_2^2\right]}\\
=&	\sqrt{\E \left[ \norm{{\bf A}+{\bf C}+{\bf G}+{\bf H}}_2^2\right]} \\
\le & \sqrt{\E\left[\norm{{\bf A}+{\bf G}}_2^2\right]}+ \sqrt{\E\left[\norm{{\bf C}}_2^2\right]} + \sqrt{\E\left[\norm{{\bf H}}_2^2\right]}.
\end{aligned}
\end{flalign}

Below, we estimate the three terms in the right-hand side separately.

\begin{enumerate}
	\item[1.] 	Define $\rho = \sqrt{ \tau^2+h \kappa^2}$, where
	\begin{flalign*} \tau^2 = 
	\begin{cases}
	(1-mh)^2 +\frac{h\delta(4M+\delta)}{2(m+M)}, \text{ for } h \in (0, \frac{2}{m+M});\\
	(1-Mh)^2+\frac{h\delta(4M +\delta)}{2(m+M)},  \text{ for } h \in (\frac{2}{m+M}, \frac{2}{M}).
	\end{cases}
	\end{flalign*}
	One can check that $\rho <1$ because of $\tilde{ \kappa}^2 < 2m$ and $h  < \min\left(\frac{2m-\tilde{ \kappa}^2}{m^2},\frac{2M-\tilde{ \kappa}^2}{M^2}\right) $.
	Therefore, by Proposition \ref{prop:Baillon-Haddad_extend_weak}, we have 
	\begin{flalign}\label{eqn:main_proof_part_A}
	\begin{aligned}
		\E\left[\norm{{\bf A}}_2^2\right]
	=& \E\bigg[\norm{\nabla\phi({\bf L}_0) -\nabla\phi({\bf X}_k)}_2^2 + h^2\norm{\nabla f({\bf L}_0) - \nabla f({\bf X}_k)}_2^2 \\
	& \ \ \ \  - 2h \langle\nabla f({\bf L}_0) - \nabla f({\bf X}_k), \nabla\phi({\bf L}_0) - \nabla\phi({\bf X}_k)	\rangle\bigg] \\
	\le & \E \bigg[ \left(1-\frac{h(4mM -4M\delta - \delta^2)}{2(m+M)}\right)\norm{\nabla\phi({\bf L}_0) -\nabla\phi({\bf X}_k)}_2^2 \\
	& \ \ \ \ + h\left(h-\frac{2}{m+M}\right)\norm{\nabla f({\bf L}_0) - \nabla f({\bf X}_k)}_2^2 	\bigg] \\
	\le &  \tau^2 W_{2,\phi}^2(\mu_k, \pi).
	\end{aligned}
	\end{flalign}
	
	The last inequality is derived from \eqref{assumption:A4} 
	if $h \in \left(\frac{2}{m+M}, \frac{2}{M}\right)$ or \eqref{assumption:A3} 
	  if $h \in \left(0, \frac{2}{m+M}\right)$.
	
	On the other hand, from It\^{o}'s isometry (Lemma \ref{lemma:Ito's_isometry}) and  assumption \eqref{assumption:A1}, we have
	\begin{flalign}\label{eqn:main_proof_part_G}
	\begin{aligned}
	\E[\norm{{\bf G}}_2^2]
	= & \E\left[ h	\norm{\sqrt{2}\left([D^2\phi({\bf L}_0)]^{\frac{1}{2}} - [D^2\phi({\bf X}_k)]^{\frac{1}{2}}	\right)}_F^2\right]\\
	\le &  h \E\left[	\kappa^2 \norm{\nabla\phi({\bf L}_0) - \nabla\phi({\bf X}_k)}_2^2	\right]\\
	= & h \kappa^2  W_{2,\phi}^2(\mu_k, \pi). 
	\end{aligned}
	\end{flalign}
	
	Note that $\E\left[\langle {\bf A}, {\bf G} \rangle \right] =0$, since ${\bm \xi}_{k+1}$ is independent of $({\bf X}_k, {\bf L}_0)$. Therefore, combining equations~\eqref{eqn:main_proof_part_A} and~\eqref{eqn:main_proof_part_G}, one has 
	\begin{flalign}\label{eqn:main_proof_part_A_plus_G}
	\begin{aligned}
	\sqrt{\E\left[\norm{{\bf A} +{\bf G}}_2^2\right]}  
	= \sqrt{\E\left[\norm{{\bf A}}_2^2 	+ \norm{{\bf G}}_2^2 \right]}
	\le  \sqrt{ \left(\tau^2+h \kappa^2\right)} W_{2,\phi}(\mu_k, \pi)
	=  \rho  W_{2,\phi}(\mu_k, \pi).
	\end{aligned}
	\end{flalign}

\item[2.] Applying Minkowski's integral inequality (Lemma \ref{lemma:Minkowski_inequality}), assumption \eqref{assumption:A4}, and Proposition \ref{prop:main_proof_phi}, 
\begin{align*}
\sqrt{\E\left[\norm{{\bf C}}_2^2\right]}
&\le  \int_0^h \sqrt{\E \left[\norm{\nabla f({\bf L}_0) - \nabla f ({\bf L}_s)}_2^2\right]} ds\\
&\le M \int_0^h \sqrt{\E \left[\norm{\nabla \phi({\bf L}_0) - \nabla \phi ({\bf L}_s)}_2^2\right]} ds \\
&\le M \int_0^h \left(s\sqrt{MpR} + \sqrt{2spR}\right) ds \\
&\le  \frac{7\sqrt{2}}{6}M h^{\frac{3}{2}} p^{\frac{1}{2}}R^{\frac{1}{2}} .	
\end{align*}

\item[3.] By It\^{o}'s isometry, assumption 
\eqref{assumption:A1},  and Proposition \ref{prop:main_proof_phi}, 
\begin{align*}
\E\left[ \norm{{\bf H}}_2^2\right]
&=	\int_0^h \E\left[   \norm{ \sqrt{2} \left([D^2\phi({\bf L}_s)]^{\frac{1}{2}} - [D^2\phi({\bf L}_0)]^{\frac{1}{2}}\right)}_F^2  \right] ds \\
&\le  \kappa^2 \int_0^h \E\left[   \norm{ \nabla\phi({\bf L}_s) - \nabla\phi({\bf L}_0)}_2^2  \right] ds \\
&\le 	\kappa^2 \int_0^h \left( s\sqrt{MpR} + \sqrt{2spR}\right)^2 ds\\
& \le  \kappa^2  h^2 pR \left(1+ \sqrt{\frac{M}{3}}h^{\frac{1}{2}}\right)^2.
\end{align*}
\end{enumerate}

In conclusion, combining \eqref{eqn:W2_total} and the above, we arrive at 
\begin{align*}
W_{2,\phi}(\mu_{k+1}, \pi) 
\le & \sqrt{\E\left[\norm{{\bf A}+{\bf G}}_2^2\right]}+ \sqrt{\E\left[\norm{{\bf C}}_2^2\right]} + \sqrt{\E\left[\norm{{\bf H}}_2^2\right]} \\
&\le   \rho  W_{2,\phi}(\mu_k, \pi) +  \frac{7\sqrt{2}}{6}M h^{\frac{3}{2}} p^{\frac{1}{2}}R^{\frac{1}{2}} + \kappa h p^{\frac{1}{2}}R^{\frac{1}{2}} + \sqrt{\frac{M}{3}}\kappa h^{\frac{3}{2}} p^{\frac{1}{2}}R^{\frac{1}{2}} \\
& =   \rho  W_{2,\phi}(\mu_k, \pi) +  h p^{\frac{1}{2}}\beta_1(R, \kappa) + h^{\frac{3}{2}}p^{\frac{1}{2}} \beta_2(R, M, \kappa).
\end{align*}
\QEDA

%% file: conclusion.tex

\section*{Conclusion}

In this paper, we have proposed the first theoretical guarantees for the discretized Langevin counterpart of the celebrated mirror descent algorithm to sample from densities whose logarithms are not necessarily Lipschitz-smooth or strongly concave. We showed that it is a stable discretization of the continuous Riemannian Langevin flow, more precisely, that it contracts toward a Wasserstein ball associated with a Hessian squared Riemannian metric. 
This analysis highlights the critical role played by the self-concordance of the entropy function and the relative anisotropy of the entropy and log-distribution (controlled by bounding the associated commutator). 

\section*{Acknowledgments} 

The authors are grateful to Arnak Dalalyan, Marco Cuturi, and Paul Rolland for stimulating conversations. They also thank Institut Henri Poincar\'e for 
their hospitality throughout "The Mathematics of Imaging" semester program, during which the collaboration started. The work of the first two authors is supported by the ERC project NORIA.

%% file: app_sufficient_condition.tex

\bigskip

\section{Well-posedness of \eqref{SDE_Y}}\label{app:sufficient_condition} 
Let us recall the SDE \eqref{SDE_Y},
\begin{flalign*}
d{\bf Y}_t = -\nabla f \circ \nabla \phi^*({\bf Y}_t) dt + \sqrt{2 [D^2 \phi^*({\bf Y}_t)]^{-1}} d{\bf B}_t .
\end{flalign*}
Let $\mathcal{Y} \eqdef \nabla\phi(\Xx)$. The following conditions are usually required for existence and uniqueness of (strong) solutions to this SDE in time interval $[0,T]$ (see \citep[Theorem~5.2.1]{Oksendal03}):
\begin{itemize}
	\item  {\bf Lipschitz condition:} there exists $K_1>0$, such that for all vectors ${\bf y}_1, {\bf y}_2 \in \mathcal{Y}$ (and all $t\in [0,T]$),
	\begin{flalign}\label{eq:lipcond}
	\sqrt{2}\norm{D^2 \phi^*({\bf y}_1)^{-\frac{1}{2}} - D^2 \phi^*({\bf y}_2)^{-\frac{1}{2}}}_F + \norm{\nabla f (\nabla \phi^*({\bf y}_1)) - \nabla f(\nabla \phi^*({\bf y}_2))}_2 \le K_1 \norm{{\bf y}_1 - {\bf y}_2}_2.
	\end{flalign}
	Let ${\bf x}_i = \nabla \phi^*({\bf y}_i)$ for $i =1,2$. Then the above inequality is equivalent to, for all ${\bf x}_1, {\bf x}_2 \in \mathcal{X}$,
	\begin{flalign*}
	\sqrt{2}\norm{D^2 \phi({\bf x}_1)^{\frac{1}{2}} - D^2 \phi({\bf x}_2)^{\frac{1}{2}}}_F + \norm{\nabla f ({\bf x}_1) - \nabla f ({\bf x}_2)}_2 \le K_1 \norm{\nabla \phi({\bf x}_1) - \nabla \phi({\bf x}_2)}_2.
	\end{flalign*}
	In view of assumptions~\eqref{assumption:A1} and~\eqref{assumption:A4}, the Lipschitz condition \eqref{eq:lipcond} holds with $K_1=M+\kappa$.
	\item {\bf Growth condition:}  there exist $K_2 >0$, such that for all ${\bf y} \in \mathcal{Y}$ (and $t \in [0,T]$),
	\begin{flalign}\label{eq:growthcond}
	2\norm{[D^2 \phi^*({\bf y})]^{-\frac{1}{2}}}^2_F + \norm{\nabla f \circ \nabla\phi^*({\bf y})}_2^2 \le K_2 (1+ \norm{{\bf y}}_2^2).
	\end{flalign}
	Similarly, this is equivalent to the existence of $K_2 >0$ such that for all ${\bf x} \in \mathcal{X}$,
	\begin{flalign*}
	2\norm{[D^2 \phi({\bf x})]^{\frac{1}{2}}}^2_F + \norm{\nabla f ({\bf x})}_2^2 \le K_2 (1+ \norm{\nabla \phi({\bf x})}_2^2).
	\end{flalign*}
	Again, owing to~\eqref{assumption:A1} and~\eqref{assumption:A4}, 
	one easily sees that \eqref{eq:growthcond} holds with $K_2$ depending on $M$ and $\kappa$.
\end{itemize}

\begin{remark}
Although the Lipschitz and Growth conditions are general requirements to guarantee the existence and uniqueness of solutions to SDE~\eqref{SDE_Y}, one can easily check that the Lipschitz condition implies the other one.
\end{remark}

\begin{remark}
Examples of entropies $\phi$ verifying for instance \eqref{assumption:A1} are given in the text, e.g., Burg's entropy $\phi(x) = -\log(x)$ on $\R_{++}$. However, this does hold for the Boltzmann-Shannon $\phi(x) = x\log(x)$ on $\R_{++}$.
\end{remark}

%% file: app_assumptions.tex

\medskip

\section{Assumption \eqref{assumption:A3} v.s. relative strong convexity; and \eqref{assumption:A4} v.s. relative smoothness}\label{app:assumptions}
Throughout, $f$ and $\phi$ are assumed $C^2$ on $\mathcal{X}$. By Cauchy-Schwarz inequality, \eqref{assumption:A3} implies \begin{flalign}\label{eqn:A1_weak}
\exists m \ge 0, \text{ s.t. } m \norm{\nabla \phi({\bf x}) - \nabla\phi({\bf x}')}_2 \le \norm{\nabla f ({\bf x}) - \nabla f({\bf x}')}_2 \ \ \forall {\bf x}, {\bf x}' \in \mathcal{X}.
\end{flalign}

Since $\mathcal{X}$ is open, for any ${\bf x} \in \mathcal{X}$,~\eqref{eqn:A1_weak} and \eqref{assumption:A4} implies that for all ${\bf{z}} \in \RR^p$ and $t$ sufficiently small
\[
m \norm{\nabla\phi({\bf x}+t{\bf z}) - \nabla\phi({\bf x})}_2 \le
\norm{\nabla f({\bf x}+t{\bf z}) -\nabla f ({\bf x})}_2 \le M \norm{\nabla\phi({\bf x}+t{\bf z}) - \nabla\phi({\bf x})}_2.
\]
Dividing by $t$ and passing to the limit as $t \to 0^+$, we get
\[
m \norm{D^2\phi({\bf x}){\bf z}}_2 \le
\norm{D^2 f({\bf x}){\bf z}}_2 \le M \norm{D^2\phi({\bf x}){\bf z}}_2, \quad \forall {\bf x} \in \mathcal{X}, \forall {\bf z} \in \RR^p.
\]

Squaring, this is equivalent to
\begin{equation}\label{eq:relsmoothpds}
m^2 \dotp{(D^2 \phi({\bf x}))^2{\bf z}}{{\bf z}} \le
\dotp{(D^2 f({\bf x}))^2{\bf z}}{{\bf z}} \le M^2 \dotp{(D^2 \phi({\bf x}))^2{\bf z}}{{\bf z}}, \quad \forall {\bf x} \in \mathcal{X}, \forall {\bf z} \in \RR^p ,
\end{equation}
or
\begin{equation}\label{eq:relsmoothsq}
(m D^2 \phi({\bf x}))^2 \preceq (D^2 f({\bf x}))^2 \preceq (M D^2 \phi({\bf x}))^2, \quad \forall {\bf x} \in \mathcal{X}.
\end{equation}
where $\preceq$ is the Loewner order defined by the cone of positive semi-definite matrices. We recall the following lemma due to \citep[Theorem~1]{Stepniak87}.
\begin{lemma}\label{lem:psdsqrt}
	For any positive semidefinite matrices ${\bf A}$ and ${\bf B}$, if ${\bf A}^2 \succeq {\bf B}^2$, then ${\bf A} \succeq {\bf B}$.
\end{lemma}
Applying this lemma with ${\bf A}=MD^2 \phi({\bf x})$ and ${\bf B}=D^2 f({\bf x})$, and then with ${\bf A}=D^2 f({\bf x})$ and ${\bf B}=mD^2 \phi({\bf x})$, we conclude that \eqref{eq:relsmoothsq} implies 
\begin{equation}\label{eq:relsmooth}
m D^2 \phi({\bf x}) \preceq D^2 f({\bf x}) \preceq M D^2 \phi({\bf x}) , \quad \forall {\bf x} \in \mathcal{X} .
\end{equation}
According to \citep[Proposition~1.(i, ii)]{BauschkeBolteTeboulle17}, \eqref{eq:relsmooth} is equivalent to smoothness and strong convexity of $f$ relatively to $\phi$, as defined in \cite{LuFreundNesterov17}. 

Overall, we have proved the following claim.
\begin{proposition}
Suppose that $f$ and $\phi$ are $C^2(\Xx)$. Then \eqref{assumption:A3} implies $m$-strong relative convexity with respect to $\phi$ and \eqref{assumption:A4} implies $M$-relative smoothness of $f$ with respect to $\phi$, i.e.~\eqref{eq:relsmooth} holds.
\end{proposition}
Observe that the converse implication in Lemma~\ref{lem:psdsqrt} does not hold in general, see \cite{Stepniak87}, and thus \eqref{eq:relsmooth} $\not\Rightarrow$ \eqref{eq:relsmoothsq} in general. In turn assumptions \eqref{assumption:A3} and \eqref{assumption:A4} are strictly stronger than relative smoothness and strong convexity.

%% file: app_Baillon_Haddad.tex

\medskip

\section{Proof of a stronger version of Proposition \ref{prop:Baillon-Haddad_extend_weak}, the Baillon-Haddad type inequality}\label{app:Baillon-Haddad_extend}
In this section, we will prove a Baillon-Haddad type inequality, as in Proposition \ref{prop:Baillon-Haddad_extend_weak}, but with weaker assumptions. This inequality serves as an essential step in the proof of Theorem \ref{thm:contractibility}.

In the following, we denote by $\mathcal{M}_{l\times n}$ the space of all matrices that have $l$ rows and $n$ columns and whose entries have real values.

\medskip

\begin{lemma}[{\citep[Example~5.6.6]{HornJohnson12}}]
	For any matrix ${\bf M} \in \mathcal{M}_{l\times n}$, $\norm{{\bf M}}_2 = \max_{{\bf v}\in \R^n} \frac{\norm{{\bf M}{\bf v}}_2}{\norm{{\bf v}}_2}$.
\end{lemma}

\medskip

\begin{remark}
	From the above lemma, it is clear that $\norm{{\bf M}_1{\bf M}_2}_2 \le \norm{{\bf M}_1}_2 \norm{{\bf M}_2}_2$ for any ${\bf M}_1 \in \mathcal{M}_{k\times l}$ and ${\bf M}_2 \in \mathcal{M}_{l\times n}$.
\end{remark}

\medskip

\begin{definition}[Contractibility] We say a domain $\mathcal{U} \subset \R^p$ is contractible if there exists some point ${\bf c} \in \mathcal{U}$ such that
	the constant map ${\bf x}\mapsto {\bf c}$ is homotopic to the identity map on $\mathcal{U}$.
\end{definition}

\medskip

\begin{definition}[Differential Forms]
	Let $0 \le k \le p$. A differential $k$-form $g: \mathcal{U} \rightarrow \Lambda^{k}$ will be written as $g = \sum_{1\le i_1 < \cdots < i_k \le p} g_{i_1\cdots i_k} dx^{i_1}\wedge \cdots \wedge dx^{i_k}$, where $g_{i_1\cdots i_k}: \mathcal{U} \rightarrow \R$ for every $1\le i_1 < \cdots < i_k \le p$ and $\Lambda^{k} = \Lambda^{k}\left({\R^{p}}^*\right)$ with ${\R^{p}}^*$ being the dual of $\R^{p}$ as a vector space. When $g_{i_1\cdots i_k} \in C^r(\mathcal{U})$ for every $1\le i_1 < \cdots < i_k \le p$, we will write $g \in C^r(\mathcal{U}; \Lambda^{k})$.
\end{definition}

\medskip

\begin{lemma}[Poincar\'{e} lemma, {\citep[Theorem~8.1]{CsatoDacorognaKneuss12}}]
	Let $r\ge 1$ and $0 \le k \le p-1 $ be integers and $\mathcal{U} \subset \R^p$ be 
	an open contractible set.
	Let $g\in C^r(\mathcal{U};\Lambda^{k+1})$ with $dg = 0$ in $\mathcal{U}$. Then there exists $G \in C^r(\mathcal{U};\Lambda^{k})$ such that $dG = g$ in $\mathcal{U}$.
\end{lemma}

\medskip
\begin{remark}
	For relaxation on the contractibility of the domain and sharper regularity in H\"{o}lder spaces, see \cite[Theorem~8.3]{CsatoDacorognaKneuss12}. 
\end{remark}

\medskip

\begin{proposition}[Baillon-Haddad extension]\label{prop:Baillon-Haddad_extend_strong}
	Assume that $\mathcal{X}$ is contractible, $\phi$ is a Legendre function on $\mathcal{X}$, $f$ and $\phi\in C^{3}(\Xx)$ satisfying \eqref{assumption:A5}, and that there exist $0 \le m \le M$ such that for any ${\bf x}_1, {\bf x}_2 \in \mathcal{X}$,
	\begin{flalign}
	m\norm{\nabla\phi({\bf x}_1)-\nabla\phi({\bf x}_2)}_2^2 
	\le \langle \nabla f({\bf x}_1) - \nabla f({\bf x}_2), \nabla\phi({\bf x}_1) -\nabla\phi({\bf x}_2) \rangle \le M \norm{\nabla\phi({\bf x}_1)- \nabla\phi({\bf x}_2)}_2^2.
	\end{flalign}
	Then for all ${\bf x}_1, {\bf x}_2 \in \mathcal{X}$, we have
	\begin{flalign}\label{Baillon-Haddad_extend_0}
	\begin{aligned}
	&\langle\nabla f({\bf x}_1) - \nabla f({\bf x}_2), \nabla\phi({\bf x}_1) - \nabla\phi({\bf x}_2)\rangle \\
	\ge & \frac{1}{m+M}\norm{\nabla f({\bf x}_1) - \nabla f({\bf x}_2)}_2^2 + \frac{4mM -4M\delta - \delta^2}{4(m+M)}\norm{\nabla\phi({\bf x}_1) - \nabla\phi({\bf x}_2)}_2^2. 
	\end{aligned}
	\end{flalign}
\end{proposition}

\medskip

\begin{remark}
	1.  Under the same assumptions as above and assuming $D^2\phi$ and $D^2 f$ are commutable, then for any ${\bf x}_1, {\bf x}_2 \in \mathcal{X}$,
	\begin{flalign}\label{Baillon-Haddad_extend_2}
	\begin{aligned}
	&\langle\nabla f({\bf x}_1) - \nabla f({\bf x}_2), \nabla\phi({\bf x}_1) - \nabla\phi({\bf x}_2)\rangle \\
	\ge & \frac{1}{m+M}\norm{\nabla f({\bf x}_1) - \nabla f({\bf x}_2)}_2^2 + \frac{mM}{m+M}\norm{\nabla\phi({\bf x}_1) - \nabla\phi({\bf x}_2)}_2^2. 
	\end{aligned}
	\end{flalign}
	2. If, in addition, $m=0$, then the inequality becomes
	\begin{flalign}\label{Baillon-Haddad_extend_3}
	\langle\nabla f({\bf x}_1) - \nabla f({\bf x}_2), \nabla\phi({\bf x}_1) - \nabla\phi({\bf x}_2)\rangle 
	\ge \frac{1}{M}\norm{\nabla f({\bf x}_1) - \nabla f({\bf x}_2)}_2^2. 
	\end{flalign}
	This is the canonical form of Baillon-Haddad inequality, which is equivalent to equation ~\eqref{Baillon-Haddad_extend_2}.\\
	3. In general, if $m=0$ (but $\delta$ may not), the inequality~\eqref{Baillon-Haddad_extend_0} implies relative Lipschitz smoothness
	\begin{flalign}\label{Baillon-Haddad_extend_4}
	\norm{\nabla f({\bf x}_1) - \nabla f({\bf x}_2)}_2 \le \left(M + \frac{\delta}{2}\right)\norm{\nabla\phi({\bf x}_1) - \nabla\phi({\bf x}_2)}_2. 
	\end{flalign}
\end{remark}

\medskip

\begin{proof}[Proof of Proposition \ref{prop:Baillon-Haddad_extend_strong}]
	Denote ${\bf A}({\bf y}): = D^2 f(\nabla\phi^*({\bf y}))$ and ${\bf B}({\bf y}): = D^2 \phi^*({\bf y})$.

	Notice that 
	\begin{flalign*}
		&d\left(\frac{1}{2}\sum_{i,j} \left[ ({\bf A}{\bf B})_{ji} - ({\bf A}{\bf B})_{ij} \right] dy_i \wedge dy_j\right)\\
	= & \sum_{i,j,l} \frac{1}{2}d\left(  \partial_{jl} f(\nabla\phi^*)  \partial_{li}\phi^* - \partial_{il} f(\nabla\phi^*)  \partial_{lj}\phi^* 	\right) \wedge dy_i \wedge dy_j\\
	= & \sum_{i,j,k,l} \frac{1}{2}\left( \partial_{jl}f(\nabla\phi^*)\partial_{lik}\phi^* -\partial_{il}f(\nabla\phi^*)\partial_{ljk}\phi^*  +\sum_{m} \partial_{jlm}f(\nabla\phi^*)\partial_{mk}\phi^*\partial_{li}\phi^* -\right.\\
	& \ \ \ \ \ \ \ \ \ \ \ \ \ \ \ \ \ \ \left.  -\sum_{m}\partial_{ilm}f(\nabla\phi^*) \partial_{mk}\phi^*\partial_{lj}\phi^* \right) d y_k\wedge dy_i\wedge dy_j\\
	= & \sum_{i,j,k,l}\frac{1}{6} \cdot 0 \ d y_k\wedge dy_i\wedge dy_j + \sum_{i,j,k,l,m}\frac{1}{6} \cdot 0\ d y_k\wedge dy_i\wedge dy_j\\
	= & 0.
	\end{flalign*}
	
	By the Poincar\'{e} lemma, there exists a $1$-form $\omega$ on $\mathcal{Y}$ such that
	\begin{flalign}\label{omega_tilde}
	d \omega = \frac{1}{2}\sum_{i,j} \left[ ({\bf A}{\bf B})_{ji} - ({\bf A}{\bf B})_{ij} \right] dy_i \wedge dy_j
	\end{flalign} 
	
	Note that $\omega$ is a $1$-form on $\mathcal{Y}$, which corresponds to a vector field ${\bf g}:\mathcal{Y}\rightarrow \R^p$ such that $\omega = {\bf g} \cdot d{\bf y}$. Define $ {\bf \tilde{g}}: = \nabla f\circ\nabla\phi^* -{\bf g} : \mathcal{Y} \rightarrow \R^p$.

	By Stokes-Cartan theorem, for any $\mathcal{U} \subset \mathcal{Y}$,  one has
	\begin{flalign*}
	 \begin{aligned}
	 \int_{\partial \mathcal{U}} \nabla f\circ \nabla \phi^* \cdot d{\bf y} 
	 = & \int_{\mathcal{U}} d\left(\sum_{j = 1} \partial_j f(\nabla \phi^*) d y_j\right) 
	 \\
	 =&  \frac{1}{2}\int_{\mathcal{U}} \sum_{i,j=1}^p \left[ ({\bf A}{\bf B})_{ji} - ({\bf A}{\bf B})_{ij} \right] dy_i \wedge dy_j\\
	 = & \int_{\mathcal{U}} d\omega = \int_{\partial \mathcal{U}} \omega
	 =  \int_{\partial \mathcal{U}} {\bf g}\cdot d{\bf y}.
	 \end{aligned}
	\end{flalign*}

	This implies, for any closed curve $\Gamma$ on $\mathcal{Y}$, one has $$\oint_{\Gamma} {\bf \tilde{g}}\cdot d{\bf y} =0.$$ 
	That is, ${\bf \tilde{g}}$ is path-independent. Define $\tilde{f}$ as a function on $\mathcal{Y}$ from any given point ${\bf y}_0\in \mathcal{Y}$ such that $\tilde{f}({\bf y}) := \tilde{f}({\bf y}_0) +\int_{\Gamma} {\bf \tilde{g}} \cdot d{\bf y}$,
	where $\Gamma$ is any smooth curve from ${\bf y}_0$ to ${\bf y}$. Therefore, 
	\begin{flalign}\label{eqn:nablatilde{f}}
		\nabla \tilde{f} = {\bf \tilde{g}} = \nabla f\circ \nabla \phi^* - {\bf g}.
	\end{flalign}
	
	From~\eqref{omega_tilde}, we know
	$\partial_i g_j = \frac{1}{2} \left[ ({\bf A}{\bf B})_{ji} - ({\bf A}{\bf B})_{ij} \right], \text{ for all } 1 \le i,j \le p.$
	Thus, \eqref{eqn:nablatilde{f}} implies
	\begin{flalign*}
	(D^2 \tilde{f})_{ji} = & \partial_i \partial_j \tilde{f} = \partial_i (\partial_j f(\nabla\phi^*) - g_j) 
	= \sum_{k} \partial_{jk}f(\nabla\phi^*)\cdot\partial_{ki}\phi^* - \partial_i g_j \\
	= & ({\bf B}{\bf A})_{ij} + \frac{1}{2}[({\bf A}{\bf B})_{ij} - ({\bf B}{\bf A})_{ij}] = \frac{1}{2}[({\bf A}{\bf B})_{ij} + ({\bf B}{\bf A})_{ij}].
	\end{flalign*}
	This shows that $D^2 \tilde{f}$ is symmetric and \begin{flalign}\label{eqn:D2tilde{f}}
		D^2 \tilde{f} = \frac{1}{2}({\bf A}{\bf B}+ {\bf B}{\bf A}) = \frac{1}{2}\left(D^2 f\circ \nabla\phi^* \cdot D^2\phi^* + D^2 \phi^* D^2 f\circ\nabla \phi^*  \right).
	\end{flalign}
	
	By assumption, there exist $0 \le m \le M$ such that for any ${\bf x}_1, {\bf x}_2 \in \mathcal{X}$,
	\begin{flalign*}
	m\norm{\nabla\phi({\bf x}_1)-\nabla\phi({\bf x}_2)}_2^2 \le \langle \nabla f({\bf x}_1) - \nabla f({\bf x}_2), \nabla\phi({\bf x}_1) -\nabla\phi({\bf x}_2) \rangle \le M \norm{\nabla\phi({\bf x}_1)- \nabla\phi({\bf x}_2)}_2^2.
	\end{flalign*}
	This implies for any ${\bf y}_1, {\bf y}_2\in \mathcal{Y}$,
	\begin{flalign*}
	m\norm{{\bf y}_1 - {\bf y}_2}_2^2 \le \langle \nabla f(\nabla\phi^*({\bf y}_1)) - \nabla f(\nabla\phi^*({\bf y}_2)), {\bf y}_1 - {\bf y}_2 \rangle \le M \norm{{\bf y}_1 - {\bf y}_2}_2^2.
	\end{flalign*}
	Thus, for any ${\bf v} \in \R^p$ and ${\bf y} \in \mathcal{Y}$, one has 
	\begin{flalign*}
	m \norm{{\bf v}}_2^2 \le {\bf v}^T \frac{[D(\nabla f \circ \nabla\phi^*)({\bf y})]^T + [D(\nabla f \circ \nabla\phi^*)({\bf y})]}{2} {\bf v} \le M \norm{{\bf v}}_2^2.
	\end{flalign*}
	This reads, from 
	\eqref{eqn:D2tilde{f}}, $$m {\bf I}_p \preceq D^2 \tilde{f} ({\bf y}) \preceq M {\bf I}_p$$ 
	for all ${\bf y}\in \mathcal{Y}$. By the classical Baillon-Haddad theorem, we know
	\begin{flalign}\label{eqn:BH_inequality}
	\langle\nabla\tilde{f}({\bf y}_1) - \nabla\tilde{f}({\bf y}_2), {\bf y}_1 - {\bf y}_2\rangle \ge \frac{1}{m+M}\norm{\nabla \tilde{f}({\bf y}_1) - \nabla \tilde{f}({\bf y}_2)}_2^2 + \frac{mM}{m+M}\norm{{\bf y}_1 - {\bf y}_2}_2^2.
	\end{flalign}
	
	Now let us estimate $\langle {\bf g}({\bf y}_1)  - {\bf g}({\bf y}_2), {\bf y}_1 - {\bf y}_2 \rangle$, $\langle\nabla \tilde{ f}({\bf y}_1) - \nabla \tilde{ f}({\bf y}_2), {\bf g}({\bf y}_1)- {\bf g}({\bf y}_2)\rangle$, and $\norm{{\bf g}({\bf y}_1)-{\bf g}({\bf y}_2)}_2^2$.
	
	\begin{enumerate}
		\item[1.] 	For any ${\bf y}_1, {\bf y}_2 \in \mathcal{Y}$, and any $t, s\in [0,1]$, denote ${\bf y}_t = t {\bf y}_1 + (1-t){\bf y}_2$ and ${\bf y}_s = s{\bf y}_1 + (1-s){\bf y}_2$. Then ${\bf g}({\bf y}_1) - {\bf g}({\bf y}_2) = \int_0^1 d\left({\bf g}({\bf y}_t)\right) = \int_0^1 \nabla {\bf g} ({\bf y}_t) \cdot ({\bf y}_1 - {\bf y}_2) dt$. Since $\nabla {\bf g}({\bf y}_t)$ is anti-symmetric,
		\begin{flalign}\label{eqn:BH_part1}
		\begin{aligned}
		\langle {\bf g}({\bf y}_1)  - {\bf g}({\bf y}_2), {\bf y}_1 - {\bf y}_2 \rangle 
		=& \int_0^1 ({\bf y}_1 - {\bf y}_2)^T [\nabla {\bf g}({\bf y}_t)]^T ({\bf y}_1 - {\bf y}_2) dt \\
		=& \frac{1}{2} \int_0^1 ({\bf y}_1 - {\bf y}_2)^T [(\nabla {\bf g}({\bf y}_t))^T+ \nabla {\bf g}({\bf y}_t)] ({\bf y}_1 - {\bf y}_2) dt =0.
		\end{aligned}
		\end{flalign}
		
		\item[2.] 
		As follows, for any $t\in [0,1]$, let 
		${\bf C}(t):  = D^2 f(\nabla\phi^*({\bf y}_t))D^2\phi^*({\bf y}_t) ={\bf A}({\bf y}_t){\bf B}({\bf y}_t)$. Then, by assumption, $\norm{{\bf C}(t)^T - {\bf C}(t)}_2 \le \delta$ for all $t\in [0,1]$.
		
		Therefore,
		\begin{flalign*}
		&\langle\nabla \tilde{ f}({\bf y}_1) - \nabla \tilde{ f}({\bf y}_2), {\bf g}({\bf y}_1)- {\bf g}({\bf y}_2)\rangle\\
		= & \sum_{l=1}^p(\partial_l \tilde{ f}({\bf y}_1) - \partial_l \tilde{ f}({\bf y}_2))\cdot(g_l ({\bf y}_1) - g_l({\bf y}_2))\\
		= & \sum_{l=1}^p \int_0^1 d(\partial_l \tilde{ f}({\bf y}_t))\cdot \int_0^1 d(g_{l}({\bf y}_s))\\
		= & \sum_{l = 1}^p  \int_0^1 \sum_{i} \partial_{il}\tilde{ f}({\bf y}_t)\cdot({\bf y}_1-{\bf y}_2)_i dt\cdot \int_0^1 \sum_{j} \partial_{j} g_l ({\bf y}_s) \cdot ({\bf y}_1 - {\bf y}_2)_j ds\\
		= & \int_0^1 \int_0^1 \sum_{i,j,l} ({\bf y}_1 - {\bf y}_2)_i \cdot \partial_{il}\tilde{ f}({\bf y}_t)\cdot \partial_j g_l  ({\bf y}_s) \cdot ({\bf y}_1 - {\bf y}_2)_j ds dt\\
		= & \int_0^1 \int_0^1 \sum_{i,j,l} ({\bf y}_1 - {\bf y}_2)_i \cdot \frac{\left({\bf A}({\bf y}_t){\bf B}({\bf y}_t)+ {\bf B}({\bf y}_t){\bf A}({\bf y}_t)\right)_{il}}{2}\\
		&\hspace{3cm}\cdot \frac{({\bf A}({\bf y}_s){\bf B}({\bf y}_s)-{\bf B}({\bf y}_s){\bf A}({\bf y}_s))_{lj}}{2} \cdot ({\bf y}_1 - {\bf y}_2)_j ds dt\\
		= & \frac{1}{4}\int_0^1 \int_0^1 ({\bf y}_1 - {\bf y}_2)^T \left[\left({\bf C}(t)+ {\bf C}(t)^T\right) \left({\bf C}(s)-{\bf C}(s)^T\right)\right] ({\bf y}_1 - {\bf y}_2) ds dt.
		\end{flalign*}
		
	\noindent	Notice that 
	\[   \norm{ \left({\bf C}(t)+ {\bf C}(t)^T\right) \left({\bf C}(s)-{\bf C}(s)^T\right)}_2 
		\le \norm{{\bf C}(t)+ {\bf C}(t)^T}_2 \norm{{\bf C}(s)-{\bf C}(s)^T}_2
		\le   2M\delta.\]
		
	\noindent Therefore, \begin{flalign}\label{eqn:BH_part2}
		\langle\nabla \tilde{ f}({\bf y}_1) - \nabla \tilde{ f}({\bf y}_2), {\bf g}({\bf y}_1)- {\bf g}({\bf y}_2)\rangle
		\le \frac{1}{4} \int_0^1 \int_0^1 2M\delta \norm{{\bf y}_1 - {\bf y}_2}_2^2 ds dt
		= \frac{1}{2}M\delta \norm{{\bf y}_1 - {\bf y}_2}_2^2.
		\end{flalign}  
		
		\item[3.] 	Similarly, one has \begin{flalign*}
		\norm{{\bf g}({\bf y}_1)-{\bf g}({\bf y}_2)}_2^2 
		= \frac{1}{4}\int_0^1 \int_0^1  ({\bf y}_1 - {\bf y}_2)^T  \left[({\bf C}(t)^T - {\bf C}(t)) ({\bf C}(s)-{\bf C}(s)^T)\right]  ({\bf y}_1 - {\bf y}_2) ds dt,
		\end{flalign*}
		\noindent and  
		\begin{flalign*}
			\norm{({\bf C}(t)^T - {\bf C}(t)) ({\bf C}(s)-{\bf C}(s)^T)}_2 
			\le   \norm{{\bf C}(t)^T - {\bf C}(t)}_2 \norm{ {\bf C}(s)-{\bf C}(s)^T}_2
			\le  \delta^2.
		\end{flalign*}
		
		\noindent Thus, \begin{flalign}\label{eqn:BH_part3}
		\norm{{\bf g}({\bf y}_1)-{\bf g}({\bf y}_2)}_2^2  \le \frac{1}{8}\int_0^1 \int_0^1 2\delta^2 \norm{{\bf y}_1- {\bf y}_2}_2^2 ds dt = \frac{\delta^2}{4} \norm{{\bf y}_1 - {\bf y}_2}_2^2.
		\end{flalign}
	\end{enumerate}

\noindent	Combining equations~\eqref{eqn:BH_inequality}-\eqref{eqn:BH_part3}, one has
	\begin{flalign*}
	&\langle\nabla f\circ \nabla\phi^*({\bf y}_1) - \nabla f\circ \nabla\phi^*({\bf y}_2), {\bf y}_1 - {\bf y}_2\rangle \\
	= & \langle\nabla\tilde{f}({\bf y}_1) - \nabla\tilde{f}({\bf y}_2), {\bf y}_1 - {\bf y}_2\rangle + \langle {\bf g}({\bf y}_1) - {\bf g}({\bf y}_2), {\bf y}_1 - {\bf y}_2\rangle \\
	= & \langle\nabla\tilde{f}({\bf y}_1) - \nabla\tilde{f}({\bf y}_2), {\bf y}_1 - {\bf y}_2\rangle \\
	\ge & \frac{1}{m+M}\norm{\nabla \tilde{f}({\bf y}_1) - \nabla \tilde{f}({\bf y}_2)}_2^2 + \frac{mM}{m+M}\norm{{\bf y}_1 - {\bf y}_2}_2^2\\
	\ge  & \frac{1}{m+M}\norm{\nabla f\circ \nabla\phi^*({\bf y}_1) - \nabla f\circ \nabla\phi^*({\bf y}_2)}_2^2 - \frac{1}{m+M}\norm{ {\bf g}({\bf y}_1)- {\bf g}({\bf y}_2)}_2^2 \\
	& - \frac{2}{m+M}\langle\nabla \tilde{ f}({\bf y}_1) - \nabla \tilde{ f}({\bf y}_2), {\bf g}({\bf y}_1)- {\bf g}({\bf y}_2)\rangle + \frac{mM}{m+M}\norm{{\bf y}_1 - {\bf y}_2}_2^2\\
	\ge & \frac{1}{m+M}\norm{\nabla f\circ \nabla\phi^*({\bf y}_1) - \nabla f\circ \nabla\phi^*({\bf y}_2)}_2^2 + \frac{4mM -4M\delta - \delta^2}{4(m+M)}\norm{ {\bf y}_1- {\bf y}_2}_2^2. 
	\end{flalign*}
	By change of variables, this implies
	\begin{flalign}
	\begin{aligned}
	&\langle\nabla f({\bf x}_1) - \nabla f({\bf x}_2), \nabla\phi({\bf x}_1) - \nabla\phi({\bf x}_2)\rangle \\
	\ge & \frac{1}{m+M}\norm{\nabla f({\bf x}_1) - \nabla f({\bf x}_2)}_2^2 + \frac{4mM -4M\delta - \delta^2}{4(m+M)}\norm{\nabla\phi({\bf x}_1) - \nabla\phi({\bf x}_2)}_2^2. 
	\end{aligned}
	\end{flalign}
\end{proof}

%% file: app_proof_of_cors.tex

\medskip

\section{Proof of Proposition \ref{cor:constant_step_size}, Corollary \ref{cor:contraction_ball}, and Proposition \ref{prop:main_proof_phi}}\label{app:proof_of_corollaries}

\medskip
\begin{proof}[Proof of Proposition \ref{cor:constant_step_size}]
	From Theorem \ref{thm:contractibility}, one has 
	\begin{flalign*}
	W_{2,\phi}(\mu_{k}, \pi) \le & \rho  W_{2,\phi}(\mu_{k-1}, \pi) +  hp^{\frac{1}{2}}\beta_1 + h^{\frac{3}{2}}p^{\frac{1}{2}}\beta_2 \\
	\le & \rho \cdot(\rho  W_{2,\phi}(\mu_{k-2}, \pi) + hp^{\frac{1}{2}}\beta_1 + h^{\frac{3}{2}}p^{\frac{1}{2}}\beta_2)  + hp^{\frac{1}{2}}\beta_1 + h^{\frac{3}{2}}p^{\frac{1}{2}}\beta_2\\
	\le & \cdots \\
	\le & \rho^{k} W_{2,\phi}(\mu_{0}, \pi) + (hp^{\frac{1}{2}}\beta_1 + h^{\frac{3}{2}}p^{\frac{1}{2}}\beta_2) (1 + \rho + \dots + \rho^{k-1}) \\
	= & \rho^{k}  W_{2,\phi}(\mu_{0}, \pi) + (hp^{\frac{1}{2}}\beta_1 + h^{\frac{3}{2}}p^{\frac{1}{2}}\beta_2)\cdot\frac{1-\rho^k}{1-\rho}\\
	< & \rho^{k}  W_{2,\phi}(\mu_{0}, \pi) + \frac{hp^{\frac{1}{2}}\beta_1 + h^{\frac{3}{2}}p^{\frac{1}{2}}\beta_2}{1-\rho}.
	\end{flalign*}
	The last inequality holds because $0<\rho <1$.
\end{proof}

\medskip

\begin{lemma}[{\citep[Lemma~1]{Chung1954}}]\label{lem:chung}
Let $\{w_k\}_{k\in\N}$ be a sequence of real numbers such that, for all $k$,
\begin{flalign}\label{eqn:chung_condition}
	w_{k+1} \leq \pa{1-\frac{c}{k}}w_{k} + \frac{c_1}{k^{s+1}},
\end{flalign}
where $c > s > 0$, $c_1 > 0$. Then for any $k$,
\begin{flalign}\label{eqn:chung_consequence}
	w_{k} \leq c_1(c-s)^{-1} k^{-s} + o(k^{-s}).
\end{flalign}
\end{lemma}

\begin{remark}\label{rmk:Chung'sLemma}
	The same consequence~\eqref{eqn:chung_consequence} holds if $c_1$ is replaced by $c_1 + o(1)$.
\end{remark}

\medskip

\begin{proof}[Proof of Corollary \ref{cor:contraction_ball}]
	1. For any $0<b_1<\frac{2m-\tilde{ \kappa}^2}{2}$, there exists $a_1>0$  
	such that $h_k=\frac{a_1}{k}$ is small enough and
	\[
	\rho_{k} \le 1-b_1h_k
	\]
	for all $k \in \N$. 
	Thus, from Theorem~\ref{thm:contractibility}, we get
	\begin{align}\label{eq:Wineq}
	\begin{aligned}
	W_{2,\phi}(\mu_{k+1}, \pi) 
	&\leq \rho_{k+1}W_{2,\phi}(\mu_{k}, \pi) + \beta_2p^{1/2}h_{k+1}^{3/2} + \beta_1p^{1/2}h_{k+1} \\
	&\leq (1-b_1h_{k+1})W_{2,\phi}(\mu_{k}, \pi) + p^{1/2}(\beta_1 + o(1))h_{k+1} .
	\end{aligned}
	\end{align}
	For any $0< s < a_1b_1$, set $w_k \eqdef h_{k+1}^sW_{2,\phi}(\mu_{k}, \pi)$. Multiplying both sides of \eqref{eq:Wineq} by $h_{k+2}^s$, and using the fact that $\{h_k\}_{k\in\N}$ is a decreasing sequence, we get
	\begin{align}\label{eq:Wineq2}
	\begin{aligned}
	w_{k+1} 
	&\leq \pa{1-\frac{a_1b_1}{k+1}}w_{k} + \frac{a_1^{s+1}p^{1/2}(\beta_1 + o(1))}{(k+1)^{s+1}} .
	\end{aligned}
	\end{align}
	Applying Lemma~\ref{lem:chung} with its Remark~\ref{rmk:Chung'sLemma}, we have
	\[
	w_{k} \leq a_1^{s+1}p^{1/2}(\beta_1 + o(1))(a_1b_1-s)^{-1} (k+1)^{-s} + o((k+1)^{-s}) .
	\]
	From the definition of $w_k$, we deduce that
	\[
	W_{2,\phi}(\mu_{k}, \pi) \leq a_1 p^{1/2}(\beta_1 + o(1))(a_1b_1-s)^{-1} + o(1) = a_1 p^{1/2}\beta_1(a_1b_1-s)^{-1} + o(1) .
	\]
	In turn, we conclude that
	\begin{flalign*}
	\limsup\limits_{k \rightarrow \infty} W_{2,\phi}(\mu_{k}, \pi)  \le a_1 p^{1/2}\beta_1(a_1b_1-s)^{-1},
	\end{flalign*}
	for any $0<s<a_1b_1$. Taking the limit at both sides when $s\rightarrow 0$,  one has 
	\begin{flalign}
		\limsup\limits_{k \rightarrow \infty} W_{2,\phi}(\mu_{k}, \pi) \le p^{1/2}\beta_1 b_1^{-1}.
	\end{flalign} 
	This implies that $W_{2,\phi}(\mu_{k}, \pi)h^2_{k+1}$ has the order $o(h_{k+1})$ whenever $h_k=\frac{a}{k}$ for $a \in (0, a_1]$. 
	
	Now let $b = \frac{2m-\tilde{ \kappa}^2}{2}$. There exists $a \in (0, a_1]$ such that $h_k=\frac{a}{k}$ is small enough and
	\[
	\rho_{k} \le 1-bh_k + \frac{m^2}{2} h_k^2
	\]
	for all $k \in \N$. Theorem~\ref{thm:contractibility} then implies
	\begin{flalign}
		\begin{aligned}
		W_{2,\phi}(\mu_{k+1}, \pi) 
		&\leq \left(1-bh_{k+1} + \frac{m^2}{2} h_{k+1}^2\right)W_{2,\phi}(\mu_{k}, \pi) + \beta_2p^{1/2}h_{k+1}^{3/2} + \beta_1p^{1/2}h_{k+1} \\
		&\leq (1-bh_{k+1})W_{2,\phi}(\mu_{k}, \pi) +  p^{1/2}(\beta_1 + o(1))h_{k+1} .
		\end{aligned}
	\end{flalign}
	Repeating the above argument by using Remark~\ref{rmk:Chung'sLemma} gives $\limsup\limits_{k \rightarrow \infty} W_{2,\phi}(\mu_{k}, \pi) \le p^{1/2}\beta_1 b^{-1} = r_0$ as claimed.	
	\QEDB
	\bigskip
	
	2. Define a function $r: [0, \infty) \rightarrow \R$ such that $r(0) = r_0$ and for all $t > 0$,
	\begin{flalign}
	r(t): = \frac{t\alpha_1 + t^{\frac{3}{2}}\alpha_2}{1-\sqrt{(1-mt)^2+\tilde{ \kappa}^2t}}.
	\end{flalign} 
	One can check that its derivative $r'(t)>0$ for all $0<t<\min\left(\frac{2}{m+M}, \frac{2m-\tilde{ \kappa}^2}{m^2}\right)$ and $\lim\limits_{t\rightarrow 0^{+}}r(t) = r_0$. If $\mu_k \notin \overline{\mathcal{B}}_{r_0}(\pi)$, i.e., $W_{2,\phi}(\mu_{k}, \pi)  > r_0$, by the continuity of $r$ at $0$, there exists $0<h_{k+1} < \min\left(\frac{2m-\tilde{ \kappa}^2}{m^2},\frac{2M-\tilde{ \kappa}^2}{M^2}, \frac{2}{m+M}\right)$ such that $W_{2,\phi}(\mu_{k}, \pi) > r(h_{k+1}) = \frac{h_{k+1}\alpha_1+ h_{k+1}^{\frac{3}{2}}\alpha_2}{1-\rho_{k+1}}$. For the $\mu_{k+1}$ obtained from the algorithm~\eqref{discrete_X}, by 
	Theorem \ref{thm:contractibility}, we know
	\begin{flalign}
	\begin{aligned}
	W_{2,\phi}(\mu_{k+1}, \pi) \le 
	&\rho_{k+1}  W_{2,\phi}(\mu_{k}, \pi)  
	+ h_{k+1} \alpha_1 + h_{k+1}^{\frac{3}{2}} \alpha_2 \\
	< &\rho_{k+1} W_{2,\phi}(\mu_{k}, \pi)  + (1-\rho_{k+1})W_{2,\phi}(\mu_{k}, \pi) \\
	= & W_{2,\phi}(\mu_{k}, \pi). 
	\end{aligned}
	\end{flalign}
	That is, the distance is strictly decreasing.\QEDB
	\bigskip
	
	3. If  $\mu_k \in \mathcal{B}_{r_0}(\pi)$, the function 
	\[\sqrt{(1-mt)^2+\tilde{ \kappa}^2t}\left(W_{2,\phi}(\mu_{k}, \pi)  - r_0\right) + t \alpha_1 + t^{\frac{3}{2}}\alpha_2\]
	is continuous in $t$ and negative at $t=0$. Thus there exists \[0<h_{k+1} < \min\left(\frac{2m-\tilde{ \kappa}^2}{m^2},\frac{2M-\tilde{ \kappa}^2}{M^2}, \frac{2}{m+M}\right)\] 
	such that \(\rho_{k+1}\left(W_{2,\phi}(\mu_{k}, \pi)  - r_0\right) + h_{k+1} \alpha_1 + h_{k+1}^{\frac{3}{2}}\alpha_2 <0\). Therefore, by 
	Theorem \ref{thm:contractibility}, we know
	\begin{flalign}
	W_{2,\phi}(\mu_{k+1}, \pi) 
	\le \rho_{k+1}  W_{2,\phi}(\mu_{k}, \pi) + h_{k+1} \alpha_1 + h_{k+1}^{\frac{3}{2}} \alpha_2 	< \rho_{k+1} r_0 < r_0.
	\end{flalign}
	That is, $\mu_{k+1} \in \mathcal{B}_{r_0}(\pi)$.
	\QEDB
	\bigskip
	
	4. Suppose $W_{2,\phi}(\mu_{k}, \pi)  = r_0$. For any $r> r_0$, there exists \[0<h_{k+1} < \min\left(\frac{2m-\tilde{ \kappa}^2}{m^2},\frac{2M-\tilde{ \kappa}^2}{M^2}, \frac{2}{m+M}\right)\] 
	such that $r > r(h_{k+1}) = \frac{h_{k+1}\alpha_1+ h_{k+1}^{\frac{3}{2}}\alpha_2}{1-\rho_{k+1}}$. Therefore, by 
	Theorem \ref{thm:contractibility}, we know
	\begin{flalign}
	W_{2,\phi}(\mu_{k+1}, \pi) 
	\le \rho_{k+1}  W_{2,\phi}(\mu_{k}, \pi) + h_{k+1} \alpha_1 + h_{k+1}^{\frac{3}{2}} \alpha_2 	< \rho_{k+1} r_0 + (1-\rho_{k+1}) r < r.
	\end{flalign}
	That is, $\mu_{k+1} \in \mathcal{B}_{r}(\pi)$.
\end{proof}

\medskip

The following lemma comes from \citep[Theorem~7.4.1.4]{HornJohnson12}.

\begin{lemma}\label{lemma:symmetric_trace}
	For any symmetric matrix ${\bf M}$ with rank $p$, we have ${\rm Tr}({\bf M}) \le p \norm{{\bf M}}_2$.
\end{lemma}

\medskip

The remark below follows clearly from the definition of the spectral norm. 
\begin{remark}\label{lemma:spectral_norm_symmetric_matrix}
	If ${\bf M}$ is  a symmetric matrix, then $\norm{{\bf M}}_2 = \lambda_{\max}({\bf M})$. 
\end{remark}

\medskip

\begin{proof}[Proof of Proposition \ref{prop:main_proof_phi}]
		Firstly, we want to show
		\begin{flalign}\label{eqn:R_1}
		\E_{{\bf L}\sim \pi}\left[ \norm{\nabla f({\bf L}) }_2^2 \right] = \E_{{\bf L}\sim \pi}\left[ {\rm Tr} (D^2f({\bf L})) \right] \le p \cdot \E_{{\bf L}\sim \pi}\left[ \norm{D^2f({\bf L})}_2 \right]\le MpR.
		\end{flalign}
		
		For the equality in~\eqref{eqn:R_1}, from integration by parts, we have
		\begin{flalign*}
		&\E_{{\bf L}\sim \pi}\left[ \norm{\nabla f({\bf L}) }_2^2 \right] \\
		= & \int_{\mathcal{X}}\langle \nabla f({\bf x}), \nabla f({\bf x})\rangle \cdot \frac{d\pi}{d {\bf x}}({\bf x})d{\bf x} \\
		= & - \int_{\mathcal{X}} \left\langle \nabla f({\bf x}), \nabla\left( \frac{d\pi}{d {\bf x}}\right)({\bf x})\right\rangle d{\bf x}\\
		= & - \int_{\partial \mathcal{X}} \frac{d\pi}{d {\bf x}}({\bf x}) \langle\nabla f({\bf x}), {\bf n}\rangle d\mathcal{H}^{p-1}({\bf x}) + \int_{\mathcal{X}} \frac{d\pi}{d {\bf x}}({\bf x}) \Delta f({\bf x}) d{\bf x}\\
		= & \int_{\partial \mathcal{X}} \left\langle\nabla \left(\frac{d\pi}{d {\bf x}}\right)({\bf x}), {\bf n}\right\rangle d\mathcal{H}^{p-1}({\bf x}) + \E_{{\bf L}\sim \pi}\left[ {\rm Tr} (D^2f({\bf L})) \right]\\
		= & \E_{{\bf L}\sim \pi}\left[ {\rm Tr} (D^2f({\bf L})) \right].
		\end{flalign*}

		The first inequality in~\eqref{eqn:R_1} can be derived using Lemma  \ref{lemma:symmetric_trace}
		when ${\bf M} = D^2 f({\bf x})$.

		For the last inequality in~\eqref{eqn:R_1}, one only need to show $\norm{D^2 f({\bf x})}_2 \le M \norm{D^2 \phi({\bf x})}_2$ for all ${\bf x} \in \mathcal{X}$. This can be derived from assumption \eqref{assumption:A4}, as shown in Appendix \ref{app:assumptions}.
		
		Secondly, since $\norm{{\bf M}}_F \le \sqrt{p}\norm{{\bf M}}_2$ holds for any matrix ${\bf M}$ with rank $p$, one has 
		\begin{flalign*}
		2\norm{\left[D^2\phi({\bf x})\right]^{\frac{1}{2}}}_F^2 \le 2p\norm{\left[D^2\phi({\bf x})\right]^{\frac{1}{2}}}_2^2 = 2p \cdot \lambda_{\max}(D^2\phi({\bf x})) =2 p \norm{D^2\phi({\bf x})}_2,
		\end{flalign*}
		for every \( {\bf x} \in \mathcal{X}\).
		Here the last equality comes from Remark \ref{lemma:spectral_norm_symmetric_matrix}. Thus, integrating at both sides against measure $\pi$ gives
		\begin{flalign}\label{eqn:R_2}
		\E_{{\bf L}\sim \pi}\left[    \norm{ \sqrt{2}[D^2\phi({\bf L})]^{\frac{1}{2}}}_F^2  \right] \le  2pR.
		\end{flalign}
		
		Lastly, 
		\begin{flalign}
		&\sqrt{\E \left[\norm{\nabla \phi({\bf L}_0) - \nabla \phi ({\bf L}_s)}_2^2\right]}\\
		= & \sqrt{\E \left[\norm{		\int_0^s\nabla f({\bf L}_r) dr - \sqrt{2} \int_0^s[D^2\phi({\bf L}_r)]^{\frac{1}{2}} d{\bf B}_r	}_2^2\right]}\\
		\label{eqn:main_proof_phi_1}	
		\le & 	\sqrt{\E\left[ \norm{\int_0^s\nabla f({\bf L}_r) dr}_2^2 \right]}	+ \sqrt{\E\left[ \norm{  \int_0^s \sqrt{2} [D^2\phi({\bf L}_r)]^{\frac{1}{2}} d{\bf B}_r}_2^2 \right]} \\
		\label{eqn:main_proof_phi_2}	
		= & 	\sqrt{\E\left[ \norm{\int_0^s\nabla f({\bf L}_r) dr}_2^2 \right]}	+ \sqrt{\int_0^s \E\left[    \norm{ \sqrt{2}[D^2\phi({\bf L}_r)]^{\frac{1}{2}}}_F^2  \right] dr } \\
		\label{eqn:main_proof_phi_3}	
		\le &  	\int_0^s \sqrt{\E\left[ \norm{\nabla f({\bf L}_r) }_2^2 \right]} dr	+ \sqrt{\int_0^s \E\left[    \norm{ \sqrt{2}[D^2\phi({\bf L}_r)]^{\frac{1}{2}}}_F^2  \right] dr }\\
		\label{eqn:main_proof_phi_4}
		= &  	\int_0^s \sqrt{\E\left[ \norm{\nabla f({\bf L}_0) }_2^2 \right]} dr	+ \sqrt{\int_0^s \E\left[    \norm{ \sqrt{2}[D^2\phi({\bf L}_0)]^{\frac{1}{2}}}_F^2  \right] dr } \\
		= &  s\sqrt{\E\left[ \norm{\nabla f({\bf L}_0) }_2^2 \right]}	+ \sqrt{s \E\left[    \norm{ \sqrt{2}[D^2\phi({\bf L}_0)]^{\frac{1}{2}}}_F^2  \right]  }\\
		\label{eqn:main_proof_phi_5}	
		\le &  s\sqrt{MpR} + \sqrt{2spR}.
		\end{flalign}
		Here~\eqref{eqn:main_proof_phi_1} comes from the triangular inequality; ~\eqref{eqn:main_proof_phi_2} is derived from It\^{o}'s isometry;
		\eqref{eqn:main_proof_phi_3} is obtained from Minkowski's inequality;
		\eqref{eqn:main_proof_phi_4} comes from the fact that ${\bf L}_r \sim \pi$ for all $r\ge 0$;
		and~\eqref{eqn:main_proof_phi_5} is from~\eqref{eqn:R_1} and~\eqref{eqn:R_2}. 	
\end{proof}